\documentclass[12pt]{article}
\title{Computing characteristic classes of subschemes of smooth toric varieties}
\author{
        Martin Helmer  \\ \normalsize
       Department of Mathematics,\\ \normalsize
 University of California, Berkeley\\ \normalsize
  \texttt{martin.helmer2@gmail.com}
  }
\date{\today}
\usepackage[toc,page]{appendix}
\usepackage{chngcntr}
\counterwithin{table}{section}
\usepackage{booktabs}% http://ctan.org/pkg/booktabs
\usepackage{colortbl}% http://ctan.org/pkg/colortbl
\usepackage{xcolor}% http://ctan.org/pkg/xcolor
\usepackage{graphicx}% http://ctan.org/pkg/graphicx

 \definecolor{Ftitle}{RGB}{11,46,108}
\definecolor{line}{RGB}{87,39,117}
\colorlet{tableheadcolor}{Ftitle!25} % Table header colour = 25% gray
\colorlet{tablerowcolor}{gray!10} % Table row separator colour = 10% gray

\usepackage[ pdfborderstyle={/S/U/W 1}]{hyperref}
\definecolor{MyBlue}{HTML}{210cac}
\definecolor{MyCiteColor}{HTML}{0099FF}
\definecolor{MyRed}{HTML}{3E186A}%{CC033C}
\hypersetup{%
  pdfborderstyle={/S/U/W 1},
  colorlinks=true,
  linkcolor=MyBlue,
  citecolor=MyCiteColor,
  urlcolor=MyRed,
  anchorcolor=MyBlue,
  linkbordercolor=MyRed,
}

\newcommand{\codim}{\mathrm{codim} }

\newcommand{\CC}{\mathbb{C}}
\newcommand{\Spec}{\mathrm{Spec}}
\newcommand{\Proj}{\mathrm{Proj}}

\newcommand{\Hom}{\mathrm{Hom}}
\newcommand{\ZZ}{ \mathbb{Z}}
\newcommand{\oo}{\mathcal{O}}
\newcommand{\pp}{\mathbb{P}}

\usepackage{xcolor}
\usepackage[margin=0.9in]{geometry}
\definecolor{mypurple}{HTML}{5B1280}

\usepackage{amsmath,amsthm}
\usepackage{amsfonts}
\newtheorem{theorem}{Theorem}[section]
\newtheorem{propn}[theorem]{Proposition}
\newtheorem{corr}[theorem]{Corollary}
\newtheorem{lemma}[theorem]{Lemma}

\newtheorem{defn}[theorem]{Definition}
\newtheorem{algorithm}{Algorithm}
\newtheorem{example}[theorem]{Example}

\newtheorem*{theorem*}{Theorem}

\begin{document}
\maketitle
\begin{abstract} \noindent
Let $X_{\Sigma}$ be a smooth complete toric variety defined by a fan $\Sigma$ and let $V=V(I)$ be a subscheme of $X_{\Sigma}$ defined by an ideal $I$ homogeneous with respect to the grading on the total coordinate ring of $X_{\Sigma}$. We show a new expression for the Segre class $s(V,X_{\Sigma})$ in terms of the projective degrees of a rational map specified by the generators of $I$ when each generator corresponds to a numerically effective (nef) divisor. Restricting to the case where $X_{\Sigma}$ is a smooth projective toric variety and dehomogenizing the total homogeneous coordinate ring of $X_{\Sigma}$ via a dehomogenizing ideal we also give an expression for the projective degrees of this rational map in terms of the dimension of an explicit quotient ring. Under an additional technical assumption we construct what we call a general dehomogenizing ideal and apply this construction to give effective algorithms to compute the Segre class $s(V,X_{\Sigma})$, the Chern-Schwartz-MacPherson class $c_{SM}(V)$ and the topological Euler characteristic $\chi(V)$ of $V$. These algorithms can, in particular, be used for subschemes of any product of projective spaces $\pp^{n_1} \times \cdots \times \pp^{n_j}$ or for subschemes of many other projective toric varieties. Running time bounds for several of the algorithms are given and the algorithms are tested on a variety of examples. In all applicable cases our algorithms to compute these characteristic classes are found to offer significantly increased performance over other known algorithms.
\end{abstract}

\section{Introduction}\label{section:intro}
Beginning with observations of Descartes (circa 1639) and first formalized in Euler's polyhedral formula (circa 1751) the Euler characteristic has become an important tool for the consideration of a diverse selection of mathematical problems. Modern realizations of the Euler characteristic have proved especially important in algebraic geometry and algebraic topology, enabling, among other things, the classification of orientable surfaces. In what follows, by Euler characteristic we mean the topological Euler characteristic. 

The topological Euler characteristic is also of interest in applications. For example, Huh \cite{huh2012maximum} and Rodriguez and Wang \cite{rodriguez2015maximum} apply the Euler characteristic of projective varieties to study problems of maximum likelihood estimation in algebraic statistics. Applications to string theory in physics include Aluffi and Esole \cite{aluffi2009cherntadpole} and Collinucci, Denef, and Esole \cite{collinucci2009d}. 

Let $M$ be a smooth variety and let $V$ be some subscheme of $M$. The Euler characteristic of $V$ may be obtained directly from the Chern-Schwartz-MacPherson class of $V$, $c_{SM}(V)$. More specifically, if we consider $c_{SM}(V)$ as an element of the Chow ring of $M$, $A^*(M)$, we have that $\chi(V)$ is equal to the degree of the zero dimensional component of $c_{SM}(V)$. It is this method that we shall use to obtain the Euler characteristic. In the case of subschemes of a projective space $\pp^n$ this approach has been used by several authors (e.g.\ Aluffi \cite{aluffi2003computing}, Jost \cite{jost2013algorithm}, the author \cite{Helmer2015,HelmerTCS}) to construct different algorithms which are capable of calculating Euler characteristics of complex projective varieties. 

In this note, we consider the computation of the Segre class, and the $c_{SM}$ class (and hence the Euler characteristic) of subschemes in a more general setting. More specifically we significantly generalize all of the results and algorithms presented by the author in \cite{Helmer2015,HelmerTCS} from the setting of subschemes of projective varieties to the setting of subschemes $V$ of certain smooth projective toric varieties $X_{\Sigma}$. While we only present the algorithm overviews (see \S\ref{section:SegreCSMMultiProj}) for subschemes of smooth projective toric varieties satisfying an additional technical condition these restrictions will not be imposed on the majority of the results. The main portion of this work will be several new theorems which will form the basis for algorithmic computation.

We now give a brief overview of the main results of this note. In what follows $X_{\Sigma}$ will always denote a toric variety defined by a fan $\Sigma$ and $R$ will denote the \textit{Cox ring} or \textit{total coordinate ring} of $X_{\Sigma}$. We will often also require that $X_\Sigma$ is either smooth and complete or smooth and projective. 

Let $X_{\Sigma}$ be a smooth complete toric variety and let $V=V(I)$ be a subscheme of $X_{\Sigma}$ defined by a homogeneous ideal $I$ in $R$, note that $R$ may be a multi-graded ring and by homogeneous we mean homogeneous with respect to the grading on $R$, see \S\ref{subsection:homhgeneCoordRing} for a discussion of this. We further suppose (without loss of generality) that a set of generators $f_0,\dots, f_r$ has been chosen for $I$ such that all the $f_i$ have the same multi-degree. Since all $f_i$ have the same multi-degree then we have $[V(f_i)]=\alpha \in A^1(X_{\Sigma})$ for all $i$ (where $A^1(X_{\Sigma})$ is the codimension one Chow group of $X_{\Sigma}$). 

We begin with a theorem which gives an explicit expression for the Segre class $s(V,X_{\Sigma})$ in  $A^*(X_{\Sigma})$, the Chow ring of $X_{\Sigma}$. Define a rational map $\phi:X_{\Sigma} \dashrightarrow \pp^r$ given by \begin{equation}
\phi: p \mapsto (f_0(p):\cdots: f_r(p)).\label{eq:phi_intro}
\end{equation}
Let $\displaystyle [Y_{\iota}]=\left[ {\phi^{-1}(\pp^{r-{\iota}})} \right] \in A^*(X_{\Sigma})$ where $\pp^{r-{\iota}}$ denotes a general linear subspace of dimension $r-{\iota}$ in $\pp^r$. Note that the cycle $[Y_{\iota}] $ has pure codimension $\iota$. Letting $ \omega_1^{(\iota)}, \dots, \omega_\mu^{(\iota)}$ be a basis of $A^{\iota}(X_{\Sigma})$ we may write $[Y_{\iota}]  = \sum_{ i=1}^\mu \gamma_{i}^{(\iota)} \omega_i^{(\iota)} $. Following the classical terminology of Harris \cite[Example 19.4]{harris1992algebraic} for a similar construction for rational maps between projective spaces we refer to the $\gamma_{i}^{(\iota)} $ as the \textit{projective degrees} of the rational map $\phi$. With these notations we have the following result, which will be proved in Theorem \ref{theorem:SegreMultiProj} below.\begin{theorem*} Let $V=V(I)$ be a subscheme of $X_{\Sigma}$ defined by an ideal $I=(f_0,\dots, f_r)$ which is homogeneous with respect to the grading of the total coordinate ring $R$ with $[V(f_i)]=\alpha $ for all $i$, further suppose that $\alpha$ is a numerically effective (nef) divisor. Then we have that
$$\vspace{-2mm}
 s(V,X_{\Sigma}) =1- \frac{1}{(1+\alpha)}\left(\sum_{\iota \geq 0} \frac{[Y_{\iota}]}{(1+\alpha)^{\iota} } \right).
$$ 
\end{theorem*}
The assumption that $\alpha$ is nef is equivalent to requiring that the subscheme $V$ be the intersection of divisors whose corresponding line bundles are generated by global sections, see Theorem \ref{theorem:nefGlobalSec} and the proceeding remarks for a discussion of this. A smooth complete toric variety is projective if and only if the cone generated by the nef divisors (often called the \textit{nef cone}, ${\rm Nef}(X_{\Sigma})$, of $X_{\Sigma}$) is full dimensional in ${\rm Pic} (X_{\Sigma})$ (see Proposition 6.3.24 of \cite{david2011toric}). Hence if we restrict to $X_{\Sigma}$ a smooth projective toric variety we can always be sure of finding a nef divisor $\alpha$ corresponding to the generators of the ideal $I$ defining $V$. 

The main ingredient needed to apply the result above to give an algorithm to compute the Segre class is an explicit expression for the projective degrees of the rational map $\phi$ in \eqref{eq:phi_intro}. One of the main contributions of this note is the development of such a result in the case where $X_{\Sigma}$ is a smooth projective toric variety. Roughly speaking the idea here is to construct appropriate zero dimensional ideals to compute the projective degrees by finding intersection numbers and using the quotient construction of toric varieties to move the computation to an affine space. See Theorem \ref{theorem:GeoQuoCox} for a review of the quotient construction. 

Let $X_{\Sigma}$ be a smooth projective toric variety. To simplify the statement of the result we first define several terms. Let $R=k[x_{\rho_1},\dots,x_{\rho_m}]$ be the graded total coordinate ring of $X_{\Sigma}$ (where $\Sigma(1)=\left\lbrace \rho_1,\dots, \rho_m \right\rbrace$ are the rays of $\Sigma$) and let $\tilde{R}$ be the ring $k[x_{\rho_1},\dots,x_{\rho_m}]$ without the grading so that $k^m= \Spec(\tilde{R})$. Let $W=V(J)$ be a reduced zero dimensional subscheme of $X_{\Sigma}$ consisting of $q$ points. We refer to an ideal $L_A$ in $\tilde{R}$ as a \textit{dehomogenizing ideal} for $W$ if the intersection $V(J)\cap V(L_A) $ in $k^m$ contains $q$ points.  

Let $n=\dim(X_{\Sigma}$). Again write $[Y_{\iota}]  = \sum_{ i=1}^\mu \gamma_{i}^{(\iota)} \omega_i^{(\iota)} $ where $ \omega_1^{(\iota)}, \dots, \omega_\mu^{(\iota)}$ is a basis of $A^{\iota}(X_{\Sigma})$. Since $X_{\Sigma}$ is a smooth projective toric variety, $A^1(X_{\Sigma})$ will have a basis consisting of nef divisors, see Proposition 6.3.24 of of Cox, Little and Schenck \cite{david2011toric}. Let $b_1,\dots, b_{q} \in A^1(X_{\Sigma})$ denote a fixed nef basis for $A^1(X_{\Sigma})$. Since the divisors $b_j$ are nef we may express the rational equivalence class of a point as a monomial in $b_1,\dots, b_{q}$. In particular let $\zeta=b_1^{c_1}\cdots b_{q}^{c_q}$ denote the rational equivalence class of a point in the dimension zero Chow group, $A_0(X_{\Sigma})$. Similarly we may write the basis elements $ \omega_1^{(\iota)}, \dots, \omega_\mu^{(\iota)}\in A^{\iota}(X_{\Sigma})$ as monomials in $b_1,\dots, b_{q}$. Since the $b_j$ are nef and since $\zeta$ is the class of a point then each exponent of $b_j$ appearing in $\omega_i^{(\iota)}$ must be less or equal to $c_j$, the exponent of $b_j$ in $\zeta$. Hence $\zeta$ is divisible by $\omega_i^{(\iota)}$. We refer to the class $a_i^{(\iota)}=\zeta/\omega_i^{(\iota)}$ as the \textit{complementary cycle} to $\omega_i^{(\iota)}$. For $b\in A^1(X_{\Sigma})$ let $\ell(b)$ be a general form in $R$ with $ [\ell(b)]=b \in A^*(X_{\Sigma})$. Writing $a_i^{(\iota)}=b_1^{j_1}\cdots b_q^{j_q}$ for $b_1,\dots, b_{q} \in A^1(X_{\Sigma})$ let $L_{a_i^{(\iota)}}$ be the ideal generated by $ j_1$ linear forms $\ell(b_1)$, $ j_2$ linear forms $\ell(b_2)$,$\dots$, and $ j_q$ linear forms $\ell(b_q)$. We refer to $L_{a_i^{(\iota)}}$ as the \textit{complementary ideal} associated to the cycle $\omega_i^{(\iota)}$.

The following result gives an expression for the projective degrees of a rational map as the dimension of an explicit quotient ring, this result is proved in Theorem \ref{theorem:sat_1mTMultiProj}. \begin{theorem*} With the notations above we have that the projective degrees are given by $$
\gamma_{i}^{(\iota)}=\dim_k \left( R[T]/ \left(( P_1, \dots, P_{\iota},S)+L_{a_i^{(\iota)}} +L_A \right) \right), 
$$ where the $P_{\ell}$ are general linear combinations of $f_0,\dots, f_r$, $L_{a_i^{(\iota)}}$ is the complementary ideal to $\omega_i^{(\iota)}\in A^{\iota}(X_{\Sigma})$, $L_A$ is a dehomogenizing ideal of $V( P_1, \dots, P_{\iota})\cap V\left( L_{a_i^{(\iota)}} \right)$ and $
S= 1-T\sum_{l=0}^r \vartheta_l f_l$ for general $\vartheta_l \in k$.
\end{theorem*}

%To apply this result in practice we may (or may not) want to add additional restrictions on the structure of $X_{\Sigma}$ so that we can be sure of being able to compute all projective degrees associated to subschemes of $X_{\Sigma}$. To ensure, for example, that we can always construct complementary ideals satisfying the hypothesis above to compute the projective degrees associated to a subscheme of $X_{\Sigma}$ it would be sufficient to require that we restrict to projective toric varieties $X_{\Sigma}$. In particular we would like there to exist a basis for $A^1(X_{\Sigma})$ consisting of nef divisors. Restricting to projective toric varieties gives such a basis since a smooth complete toric variety is projective if and only if the cone generated by the nef divisors (often called the \textit{nef cone}, ${\rm Nef}(X_{\Sigma})$, of $X_{\Sigma}$) is full dimensional in ${\rm Pic} (X_{\Sigma})$ (see Proposition 6.3.24 of \cite{david2011toric}). In particular we could restate the theorem above requiring that $X_{\Sigma}$ is a smooth projective toric variety and removing the requirement that $b_1,\dots, b_{q} \in A^1(X_{\Sigma})$ be nef, since this would then be automatic.

Note that the above result does not specify how to construct the required dehomogenizing ideal. Hence the other desirable ingredient for algorithmic usage is an explicit result which can be used to construct a dehomogenizing ideal in a simple and algorithmic manner for \textit{any} zero dimensional subscheme which is the intersection of nef divisors. Such a result is proved in Theorem \ref{theorem:countingPoints} below where we give a explicit expression for a \textit{general dehomogenizing ideal} for any (reduced) zero dimensional subscheme of $X_{\Sigma}$ provided that $X_{\Sigma}$ satisfies what we call the {affine codimension condition}. Explicitly we say that a toric variety $X_{\Sigma}$ satisfies the \textit{affine codimension condition} if the number of primitive collections of the fan $\Sigma$ is equal to $ m-n$ where $n=\dim(X_{\Sigma})$ and $m$ is the number of generating rays of $\Sigma(1)$ (equivalently $m$ is the number of variables in the Cox ring of $X_{\Sigma}$).
 
In this setting we give algorithms to compute the Segre class $s(V,X_{\Sigma})$, the Chern-Schwartz-MacPherson class $c_{SM}(V)$ and the Euler characteristic $\chi(V)$. A third algorithm to compute the $c_{SM}$ class in the special case where $V$ can be seen as a hypersurface in some smooth complete intersection subscheme of $X_{\Sigma}$ is also given. This third algorithm offers performance improvements in some cases by eliminating the need to perform inclusion/exclusion (Proposition \ref{propn:csm_higher_codim}), see \S\ref{subsection:CSM_Complete_int} for more details.

We note that arbitrary products of projective spaces $\pp^{n_1} \times \cdots \times \pp^{n_j}$ are projective (and hence have a basis for $A^1(X_{\Sigma})\cong {\rm Pic} (X_{\Sigma})$ consisting of nef divisors) and satisfy the affine codimension condition above. The affine codimension condition is also satisfied by many other projective toric varieties. For example, of the 124 unique smooth toric Fano fourfolds $42$ satisfy the affine codimension condition, this condition also holds for 205 of the 866 smooth Fano toric varieties with $\dim(X_{\Sigma})=5$ and for 1152 of the 7622 smooth Fano toric varieties with $\dim(X_{\Sigma})=6$ (these can be found using the \textsf{smoothFanoToricVariety} function in Macaulay2 \cite{M2}, see also {\O}bro \cite{obro2008classification}). For these varieties, since they satisfy the affine codimension condition, we may use Theorem \ref{theorem:countingPoints} to obtain a general dehomogenizing ideal and use this to compute characteristic classes for any subscheme of $X_{\Sigma}$ (which is the intersection of hypersurfaces corresponding to nef divisors). For toric varieties not satisfying the affine codimension condition we may still use the results of Theorems \ref{theorem:SegreMultiProj} and \ref{theorem:sat_1mTMultiProj} to compute the Segre class of a subscheme, however we would need to construct a dehomogenizing ideal via a different method.  

The last result proved is Theorem \ref{theorem:csm_complete_int_multi_proj}, this result provides a theoretical basis for an algorithm to compute $c_{SM}$ classes without using the expensive inclusion/exclusion procedure (Proposition \ref{propn:csm_higher_codim}) in some cases. This result, in particular, gives an explicit expression for the $c_{SM}$ class when the subscheme $V$ of $X_{\Sigma}$ can be seen as a hypersurface in a smooth complete intersection subscheme of $X_{\Sigma}$. 

The algorithm presented here to compute Segre classes of subschemes of toric varieties offers substantial performance improvements over the previous algorithm of Moe and Qviller \cite{moe2013segre} which is applicable in a setting similar to that presented here. The algorithm of \cite{moe2013segre} is a generalization of the previous algorithm of Eklund, Jost and Peterson \cite{Jost} and works by using appropriate saturations to compute the ideals of certain residual schemes and then computing their multi-degrees. The key advantage of our algorithm (Algorithm \ref{algorithm:SegreAlgMultiProj}) likely comes from the result of Theorem \ref{theorem:sat_1mTMultiProj} (combined with Theorem \ref{theorem:SegreMultiProj}) since this theorem reduces the problem of computing the Segre class of a subscheme to that of finding the number of solutions to certain zero dimensional polynomial systems; a problem for which there are many effective algorithms. See \S\ref{subsection:MoeQviller} and \S\ref{subsection:PerformenceMultiProj} for further discussions. 

In the context of computing Segre classes the recent algorithm of Harris \cite{harris2015computing} will allow for the computation of certain Segre classes in the Chow ring of projective space not encompassed by the current work, however the algorithm presented here has the advantage of working directly in the toric variety $X_{\Sigma}$, rather than via an embedding. From a practical standpoint, using the algorithm of \cite{harris2015computing} on even simple examples such as subschemes of a product of projective spaces would have a significant added cost due to the fact that one would need to work in an ambient projective space via the Segre embedding. For example considering subschemes $V$ of $\pp^3 \times \pp^3 \times \pp^3$ using the Segre embedding to compute $s(V,\pp^3 \times \pp^3 \times \pp^3)$ would require working in a polynomial ring with $4 \cdot 4 \cdot 4=64$ variables; working in the Cox ring of $\pp^3 \times \pp^3 \times \pp^3$, as we do here, gives a polynomial ring in $4+4+4=12$ variables. Since algebraic methods of all types are heavily dependent on the number of variables in the polynomial rings being considered this is a very substantial difference.

It should be noted that all algorithms presented in this note are probabilistic; this is because they involve a general choice of some scalars. More precisely, in the sense of algebraic geometry and using terminology from books such as Sommese and Wampler \cite{sommese2005numerical}, Algorithms \ref{algorithm:SegreAlgMultiProj}, \ref{algorithm:csm_InExMultiProj}, and \ref{algorithm:csm_Complete_Int_multi_proj} are probability one algorithms as they will return the correct result for any choice of objects within an open dense Zariski set in the associated parameter space. In practice, however, we need to make random choices from finite (albeit large) sets of either integers or rational numbers for computer implementations. A detailed probabilistic analysis of the projective case of these algorithms is carried out by the author in \cite[\S 3.2, \S 3.4]{HelmerTCS}. Let $\mathcal{S}$ denote a finite subset of our coefficient field $k$ from which we choose random elements in our algorithm. In \cite[Proposition 3.6]{HelmerTCS} a probability bound is given on the number elements needed in the set $\mathcal{S}$ to ensure a probability of failure less then a given value. In particular the probability of failure may be made arbitrarily small given a sufficiently large set $\mathcal{S}$. In practice, for computations of Segre classes of subschemes of $\pp^n$, choosing scalars from a set of $32749$ elements yielded no errors in over $10000$ trials. Based on experimental results and our experience we believe the more general case presented here will have similar probabilistic behaviour.

This note will be organized as follows. In \S\ref{section:setting_problem}, we begin by precisely stating the problem to be considered and the setting in which we shall work. Following this, we review several previous results and constructions which are important for this work. Previous algorithms to compute characteristics classes are also reviewed in \S\ref{subsection:MoeQviller}. 

The main results of this note are proved in \S\ref{section:mainresultsMultiProj}. In \S\ref{section:SegreCSMMultiProj} we apply the results of \S\ref{section:mainresultsMultiProj} to construct explicit algorithms to compute the Segre and Chern-Schwartz-MacPherson classes and the Euler characteristic of subschemes of $X_{\Sigma}$. The presentation in \S\ref{section:SegreCSMMultiProj} is restricted to the case where $X_{\Sigma}$ is a projective toric variety satisfying the affine codimension condition. Our algorithm to compute Segre classes of arbitrary subschemes $V$ of $X_{\Sigma}$ is given in Algorithm \ref{algorithm:SegreAlgMultiProj}. In Algorithm \ref{algorithm:csm_InExMultiProj} we present an algorithm to compute $c_{SM}(V)$ and $\chi(V)$ in the toric setting using the inclusion/exclusion property of $c_{SM}$ classes (see Proposition \ref{propn:csm_higher_codim}). In Algorithm \ref{algorithm:csm_Complete_Int_multi_proj} we present an algorithm to compute the $c_{SM}$ class of certain complete intersection subschemes of $X_{\Sigma}$ without using inclusion/exclusion, eliminating the inclusion/exclusion procedure often speeds up computation. Algorithm \ref{algorithm:csm_Complete_Int_multi_proj} is based on Theorem \ref{theorem:csm_complete_int_multi_proj}.

In \S\ref{subsection:PerformenceMultiProj} we discuss the performance of these algorithms. The running times of our test implementation on a variety of examples are given in \S\ref{subsection:runtimeTests} and are compared with those of other known algorithms where possible. In all cases the algorithms presented here offer improved performance in comparison to existing algorithms. Running time bounds for Algorithm \ref{algorithm:SegreAlgMultiProj} and Algorithm \ref{algorithm:csm_InExMultiProj} are given in \S\ref{subsection:RunTimeBoundMultiProj}.

The Macaulay2 \cite{M2} implementation of the Algorithms \ref{algorithm:SegreAlgMultiProj}, \ref{algorithm:csm_InExMultiProj}, and \ref{algorithm:csm_Complete_Int_multi_proj} can be found at \newline \url{https://github.com/Martin-Helmer/char-class-calc-toric}.

\begin{example}
Let $k$ be an algebraically closed field of characteristic zero and let $V=V(I)$ be the subvariety of $\pp^4\times \pp^2\cong \Proj(k[x_0,\dots,x_4]) \times \Proj(k[y_0,\dots,y_2]) $ defined by the ideal $$I=\left( 17 {x}_{0} y_0 y_2-3 {x}_{1} y_0 y_2+9 {x}_{3} y_0 y_2,5 {x}_{1} y_2^{2}+{x}_{3} y_2^{2}-3 {x}_{4}
      y_2^{2},-4 {x}_{1} y_0^{2}+7 {x}_{2} y_0^{2}+12 {x}_{3} y_0^{2}\right)$$ in $R=k[x_0,x_1,x_2,x_3,x_4,y_0,y_1,y_2]$. Also let $A^*(\pp^4 \times \pp^2) \cong \ZZ[h_1,h_2]/(h_1^{5},h_2^3)$ be the Chow ring of $\pp^4\times \pp^2$. $V$ is singular with $\codim(V)=2$ and $V$ is not a complete intersection. Using Algorithm \ref{algorithm:SegreAlgMultiProj} with input $I$ we compute the Segre class $s(V,\pp^4 \times \pp^2)$ in  $A^*(\pp^4\times \pp^2)$ to be\begin{align*}
s(V,\pp^4 \times \pp^2)= &-300 {h}_{1}^{4} {h}_{2}^{2}+40 {h}_{1}^{4} {h}_{2}+80 {h}_{1}^{3} {h}_{2}^{2}-3 {h}_{1}^{4}-12 {h}_{1}^{3} {h}_{2}+{h}_{1}^{3}
-12 {h}_{1} {h}_{2}^{2}+2
      {h}_{1} {h}_{2}+4 {h}_{2}^{2}.
\end{align*}Using Algorithm \ref{algorithm:csm_InExMultiProj} with input $I$ we obtain the Chern-Schwartz-MacPherson class \small{\begin{align*}
c_{SM}(V)= &13 {h}_{1}^{4} {h}_{2}^{2}+10 {h}_{1}^{4} {h}_{2}+22 {h}_{1}^{3} {h}_{2}^{2}+2 {h}_{1}^{4}+13 {h}_{1}^{3} {h}_{2}+18 {h}_{1}^{2} {h}_{2}^{2}+{h}_{1}^{3}+8
      {h}_{1}^{2} {h}_{2}+7 {h}_{1} {h}_{2}^{2}+2 {h}_{1} {h}_{2}+{h}_{2}^{2} 
\end{align*}}\normalsize in $A^* (\pp^4\times \pp^2 )$ and also obtain the Euler characteristic $
\chi(V)= 13.
$  \label{example:csm_ex_multi_proj}
\end{example}

\section{Setting, Review and Problem}
\label{section:setting_problem}
For the algorithmic portions of this note we primarily consider ambient spaces which are smooth projective toric varieties. For some, but not all, of the results projective can be replaced by complete. We also (primarily) consider subschemes of toric varieties over the complex numbers and take $k=\CC$. This is done because several of the results of Cox, Little, and Schenck \cite{david2011toric} which we use are stated in this setting. If the toric variety $\pp^{n_1} \times \cdots \times \pp^{n_j} $ is the ambient space we could work instead over $k$ any algebraically closed field of characteristic zero. We note that MacPherson's original construction \cite{macpherson1974chern} of the $c_{SM}$ class was over $\CC$, however subsequent constructions such as Kennedy \cite{kennedy1990macpherson} or Aluffi \cite{aluffi2006classes} require only an algebraically closed field of characteristic zero.

Let $X_{\Sigma}$ be a smooth projective toric variety of dimension $n$ with homogeneous coordinate ring $R$. Let $V=V(I)$ be any subscheme of $X_{\Sigma}$ defined by an ideal $I$ in $R$ which is homogeneous with respect to the grading on $R$. The problem we consider in this note is the following: determine the Segre class of $V$ in $X_{\Sigma}$, $s(V,X_{\Sigma})$, the Chern-Schwartz-MacPherson class of $V$, $c_{SM}(V)$, and the Euler characteristic of $V$, $\chi(V)$, in a time efficient manner on a computer algebra system. 

We will represent all characteristic classes as elements of the Chow ring of $X_{\Sigma}$, $A^*(X_{\Sigma})$. Proposition \ref{propn:ChowRingDef} gives the concrete realization of $A^*(X_{\Sigma})$ which will be used for all algorithms in this note. We abuse notation and write $s(V,X_{\Sigma})$, and $ c_{SM}(V)$ for the pushforwards to $X_{\Sigma}$ of the Segre class of $V$, and the $c_{SM}$ class of $V$, respectively. 

\subsection{The Segre Class}\label{subsection:settinSegre}

The Segre class is an important invariant in intersection theory, both because it contains important intersection theoretic information and because it can be used to construct other commonly studied structures and invariants. In particular the Chern-Fulton class (see \eqref{eq:CF_def}) and the Chern-Schwartz-MacPherson class (see Proposition \ref{propn:AluffiGeneralCSMSegre_Relation}) may be defined in terms of Segre classes.

For $V$ a proper closed subscheme of a variety $W$, we may define the Segre class of $V$ in $W$ as \begin{equation}
s(V,W)= \sum_{j\geq 1}(-1)^{j-1}\eta_*(\tilde{V}^j) =\eta_* \left( \frac{[\tilde{V}]}{1+[\tilde{V}]}\right) \in A^*(V) \label{eq:general_SegreDef}
\end{equation}
 where $\tilde{V}$ is the exceptional divisor of the blow-up of $W$ along $V$, $Bl_VW$, $\eta: \tilde{V} \to V$ is the projection (and $\eta_*$ is its pushforward), the class $\tilde{V}^k$ is the $k$-th self intersection of $\tilde{V}$, and $[\tilde{V}]$ is the class of $\tilde{V}$ in the Chow ring of the blow-up, $A^*(Bl_VW)$. See Fulton \cite[\S 4.2.2]{fulton} for further details.  
 
The total Chern class of a $j$-dimensional nonsingular variety $V$ is defined as the Chern class of the tangent bundle $T_V$; we express this as $c(V)=c(T_V) \cap [V]$ in the Chow ring of $V$, $A^*(V)$. A definition of the Chern class of a vector bundle can be found in Fulton \cite[\S3.2]{fulton}. If $V$ is a smooth projective varitey the degree of the zero dimensional component of the total Chern class of $V$ is equal to the topological Euler characteristic (this follows from the Gauss-Bonnet-Chern theorem, see, for example, \cite[Theorem 5.21]{eisenbud20163264}), that is \begin{equation}
\int c(T_V) \cap [V]=\chi(V). \label{eq:chern_euler_non_singular}
\end{equation}Here we let $\int \alpha$ denote the degree of the zero dimensional component of the class $\alpha \in A_*(V)$, that is the degree (i.e.\ the coefficient if $\alpha$ has only one term) of the part of $\alpha$ in the dimension zero Chow group $A_0(V)$, see Fulton \cite[Definition 1.4]{fulton} for more details. 

Any algorithm to compute the Segre class will immediately give us an algorithm to compute the Chern-Fulton class $c_F$ (refered to as the Canonical class by Fulton \cite{fulton}) of a subscheme $V$ of a smooth variety $M$ over an algebraically closed field. Specifically we have that \begin{equation}
c_F(V)=c(T_M) \cap s(V,M) \in A^*(M). \label{eq:CF_def}
\end{equation} The Chern-Fulton class $c_F$ is a generalization of the Chern class to singular schemes, see Fulton \cite[Examples 4.2.6, 19.1.7]{fulton}. In particular if $V$ is a smooth subscheme of $M$, any method to compute the Segre class also gives the total Chern class (of the tangent bundle), i.e.\ $$c(V)=c(T_V)\cap [V]=c(T_M) \cap s(V,M).$$ 
\subsection{The Chern-Schwartz-MacPherson Class}\label{subsection:settingCSM}
While there are several generalizations of the total Chern class to singular varieties, all of which agree with $c(T_V) \cap [V]$ for nonsingular $V$, the Chern-Schwartz-MacPherson class is the only generalization which satisfies a property analogous to (\ref{eq:chern_euler_non_singular}) for any $V$, i.e. \begin{equation}
\int c_{SM}(V)=\chi(V). \label{eq:csm_euler}
\end{equation} 

Here we briefly recall the functorial definition of $c_{SM}$ classes arising from MacPherson's results in \cite{macpherson1974chern} (see also Kennedy \cite{kennedy1990macpherson} and Aluffi \cite{aluffi2006classes}). For a scheme $V$, we take $\mathcal{C}(V )$ to be the abelian group of finite linear combinations $\sum_W \mathfrak{m}_W \mathbf{1}_W$, where $W$ are (closed) subvarieties of $V$, $\mathfrak{m}_W \in \ZZ$, and $\mathbf{1}_W$ denotes the function that is $1$ in $W$, and $0$ outside of $W$. Elements $f\in \mathcal{C}(V )$ are referred to as constructible functions and the group $\mathcal{C}(V )$ is termed the group of constructible functions on $V$. We may make $\mathcal{C}$ into a functor by letting $\mathcal{C}$ map a scheme $V$ to the group of constructible functions on $V$ and letting $\mathcal{C}$ map a proper morphism $f: V_1 \to V_2$ to $$\mathcal{C}(f)(\mathbf{1}_W)(p)=\chi(f^{-1}(p) \cap W), \;\;\; W \subset V_1, \; p\in V_2 \; \mathrm{a \; closed \; point}.$$ 

The Chow group functor $\mathcal{A}_*$ is also a functor from algebraic varieties to abelian groups. The $c_{SM}$ class arises from a natural transformation between these two functors; we abuse notation slightly and denote both the natural transformation and the associated class as $c_{SM}$. 
\begin{defn}
The Chern-Schwartz-MacPherson class is characterized by the unique natural transformation between the constructible function functor and the Chow group functor, that is $c_{SM}: \mathcal{C}\to \mathcal{A}_*$ is the unique natural transformation satisfying: \begin{itemize}
\item (\textit{Normalization}) $ c_{SM}(\mathbf{1}_V)=c(T_V) \cap [V] $ for $V $ non-singular  and complete.
\item (\textit{Naturality}) $f_{*}(c_{SM}(\phi))=c_{SM}(\mathcal{C}(f)(\phi))$, for $f:X \to Y$ a proper transformation of projective varieties, $\phi$ a constructible function on $X$. \label{defn:csm_natural_transform}
\end{itemize}
\end{defn}
In all that follows we always consider the $c_{SM}$ class as an element of the Chow group of some ambient smooth variety $M$. More precisely for $V$ a subscheme of a smooth variety $M$ we think of $c_{SM}(V)$ as an element of $A^*(M)$. Let $V_{red}$ denote the support of the scheme $V$, the notation $c_{SM}(V)$ is taken to mean $c_{SM}(\mathbf{1}_V)$ and hence, since $\mathbf{1}_V=\mathbf{1}_{V_{red}} $, we write $c_{SM}(V)=c_{SM}(V_{red}) $. 

The $c_{SM}$ class satisfies the same inclusion/exclusion relation as the Euler characteristic. That is for $V_1,V_2$ subschemes of a smooth variety $M$ we have that
 \begin{equation}
c_{SM}(V_1 \cap V_2)=c_{SM}(V_1)+c_{SM}(V_2)-c_{SM}(V_1 \cup V_2 )\;\;\; {\rm in} \; A^*(M).
\label{eq:csm_inclusion_exclusion}
\end{equation}

Note that this relation for $c_{SM}$ classes will allow us to reduce all computation of $c_{SM}$ classes to the case of hypersurfaces. From this we obtain the following proposition (see also Aluffi \cite{aluffi2003computing}).

\begin{propn}
Let $	V=X_0 \cap \cdots \cap X_{r} = V(f_0) \cap \cdots \cap V(f_r)$ be a subscheme of a smooth variety $M$ with coordinate ring $R$. Write the polynomials defining $V$ as $F=(f_0,\dots , f_r) \in R$ and let $F_{ \left\lbrace S \right\rbrace } = \prod_{i \in S} f_i $ for $S \subset \left\lbrace 1, \dots , r\right\rbrace$ . Then, $$
c_{SM}(V)= \sum_{S \subset \left\lbrace 1, \dots , r\right\rbrace} (-1)^{|S|+1}c_{SM} \left(V( F_{\left\lbrace S \right\rbrace } )\right) \;\;\; {\rm in} \; A^*(M)
$$  where $|S|$ denotes the cardinality of the integer set $S$. \label{propn:csm_higher_codim}
\end{propn}

\label{section:ch3Reveiw}

\subsection{The Graded Total Coordinate Ring} \label{subsection:homhgeneCoordRing}
In this subsection we briefly review some notation and results regarding the quotient construction of a toric variety. We will make extensive use of this construction throughout this note, for a more detailed review see, for example, Cox, Little, and Schenck \cite[\S5.1,\S5.2]{david2011toric}.

We restrict to the case where $X_{\Sigma}$ is a smooth complete toric variety defined by a fan $\Sigma \subset N\cong \ZZ^l$. Let $\Sigma(1)$ denote the set of one dimensional cones, also referred to as rays, in the fan $\Sigma$. The homogeneous coordinate ring of $X_{\Sigma}$ is given by $
R=\CC[x_{\rho} \; | \; \rho \in \Sigma(1) ]
$. $R$ is also referred to as the \textit{Cox ring} or \textit{total coordinate ring}. We grade the ring $R$ by defining the \textit{degree} of a monomial $
\mathsf{x}=\prod_{\rho \in \Sigma(1)}x_{\rho}^{a_{\rho}}
$ to be $
\deg(\mathsf{x})=\left[ V\left( \mathsf{x} \right) \right] \in A^1(X_{\Sigma}).
$ With this grading, if we set $
R_{\alpha}=\bigoplus_{\deg(\mathsf{x})=\alpha} \CC \cdot \mathsf{x}$ then we have that $R=\bigoplus_{\alpha \in A^1(X_{\Sigma})} R_{\alpha}.$ Additionally $R_{\alpha}\cdot R_{\beta}\subseteq R_{\alpha+\beta}$.

For a cone $\sigma \in \Sigma$ take $\sigma(1)=\left\lbrace \rho \in \Sigma(1) \; | \; \rho \mathrm{\; is \; a \; face \; of\;} \sigma \right\rbrace$ to be the set of one-dimensional faces of $\sigma$.
We may define the \textit{irrelevant ideal} of the coordinate ring $R$ as the ideal \begin{equation}
B= \left( \prod_{\rho \notin \sigma(1)} x_{\rho} \; | \; \sigma \in \Sigma \right) \subset R.
\end{equation}
Let $\CC^{\Sigma(1)}=\Spec(\CC[x_{\rho}\; | \; \rho \in \Sigma(1)])$ be an affine space. Define the group $G=\Hom_{\ZZ}(A^1(X_{\Sigma}),\CC^*).$  \begin{theorem}[Theorem 5.1.11 of Cox, Little, and Schenck \cite{david2011toric}]
Let $X_{\Sigma}$ be a smooth complete toric variety. Then $X_{\Sigma}=(\CC^{\Sigma(1)}-V(B))/G$ is the geometric quotient of $\CC^{\Sigma(1)}-V(B)$ by $G$. \label{theorem:GeoQuoCox}
\end{theorem}

Hence we may regard elements of $\CC^{\Sigma(1)}-V(B)$ as ``homogeneous coordinates" for $X_{\Sigma}$. We now consider the structure of $V(B)$ and $G$ in greater detail. We say the collection $v_{\rho_1}, \dots ,v_{\rho_s}$ of ray generators (i.e.\ $\rho_i=\left\langle v_{\rho_i}\right\rangle$) is a \textit{primitive collection} (see \cite[Definition 5.1.5, Proposition 5.1.6]{david2011toric}) if the collection $v_{\rho_1}, \dots ,v_{\rho_s}$ does not lie in any cone $\sigma \in \Sigma$ but every proper subset does. We may write an irreducible decomposition of $V(B)$ as \begin{equation}
V(B)= \bigcup_{ v_{\rho_1}, \dots ,v_{\rho_s} \; \mathrm{primitive}} V(x_{\rho_1}, \dots, x_{\rho_s}).\label{eq:irred_decomp_Vanishing_irel_ideal}
\end{equation}
To simplify terminology we will also refer to the collection of rays ${\rho_1}, \dots ,{\rho_s}$ as a primitive collection if the associated ray generators $v_{\rho_1}, \dots ,v_{\rho_s}$ form a primitive collection. 

\subsection{The Chow Ring of a Smooth Complete Toric Variety} \label{subsection:ChowRingSmoothToric}
Let $k=\CC$. The following proposition gives us a simple method to compute the Chow ring of a smooth, complete toric variety $X_{\Sigma}$. We use this result to compute the Chow ring $A^*(X_{\Sigma})$ in the algorithms of \S\ref{section:SegreCSMMultiProj}.
\begin{propn}[Theorem 12.5.3 of Cox, Little, and Schenck \cite{david2011toric}] Let $N$ be an integer lattice with dual lattice $\mathcal{M}$ and let $X_{\Sigma}$ be a complete and smooth toric variety with generating rays $\Sigma(1)=\rho_1,\dots, \rho_m $ where $\rho_j=\left\langle v_j \right\rangle$ for $v_j\in N$. Then the Chow ring of $X_{\Sigma}$ has the following presentation \begin{equation}
A^*(X_{\Sigma}) \cong \mathbb{Z}[x_1,\dots,x_r]/(\mathcal{I}+\mathcal{J}), 
\end{equation} with the isomorphism map specified by $[V({\rho_i})]\mapsto [x_i] $. Here $\mathcal{I}$ denotes the Stanley-Reisner ideal of the fan $\Sigma$, that is the ideal in $\mathbb{Z}[x_1,\dots,x_r]$ specified by \begin{equation}
\mathcal{I}=(x_{i_1}\cdots x_{i_s} \; | \; i_{i_j} \; \mathrm{distinct} \; \mathrm{and \;} \rho_{i_1} + \cdots + \rho_{i_s} \; \mathrm{is \; not \; a \; cone \; of \;} \Sigma ) 
\end{equation} and $\mathcal{J}$ denotes the ideal of $\mathbb{Z}[x_1,\dots,x_r]$ generated by linear relations of the rays, that is $\mathcal{J}$ is generated by linear forms $
\sum_{j=1}^r \left\langle \mathfrak{m},v_j \right\rangle x_j$ for $\mathfrak{m}$ ranging over some basis of $\mathcal{M}$.
\label{propn:ChowRingDef}
\end{propn}

In \S\ref{section:mainresultsMultiProj} we will need an additional property for the elements of $A^1(X_{\Sigma}) \cong \mathrm{Pic}(X_{\Sigma})$. Let $X$ be a normal toric variety; a Cartier divisor $D$ on $X$ is termed \textit{numerically effective} or \textit{nef} if $D\cdot C \geq 0$ for every irreducible complete curve $C\subset X$. 

\begin{theorem}[Theorem 6.3.12 of \cite{david2011toric}]
Let $D$ be a Cartier divisor on a complete toric variety $X_{\Sigma}$. $D$ is nef, that is $ D \cdot C \geq 0$ for all torus-invariant irreducible complete curves $C \subset X_{\Sigma}$, if and only if $D$ is basepoint free, i.e.\ $\oo_{X_{\Sigma}}(D)$ is generated by global sections.
\label{theorem:nefGlobalSec}
\end{theorem}

Proposition 6.3.24 of Cox, Little, and Schenck \cite{david2011toric} states that a smooth complete toric variety is projective if and only if the cone generated by the nef divisors is full dimensional in ${\rm Pic} (X_{\Sigma})$. In particular, then, when $X_{\Sigma}$ is a smooth projective toric variety there exists a basis for $A^1(X_{\Sigma})$ consisting of nef divisors.

\subsection{Previous Algorithms}\label{subsection:MoeQviller}
In \cite{moe2013segre} Moe and Qviller give an algorithm to compute the Segre class of subschemes of smooth projective toric varieties. The algorithm of Moe and Qviller \cite{moe2013segre} is based on a result which gives an expression for the Segre class of a subscheme of a smooth projective toric variety in terms of the classes in the Chow ring of certain residual sets which are computed via saturation. This result of Moe and Qviller \cite{moe2013segre} generalizes a previous result of Eklund, Jost and Peterson \cite{Jost} which gave an expression for the Segre class of a subscheme of $\pp^n$ in terms of residual sets having a similar structure. For both cases, the residual sets are in the sense of Fulton's residual intersection theorem/formula (Theorem 9.2 and Corollary 9.2.3 of Fulton \cite{fulton}). 

The main computational step of the algorithm of Moe and Qviller \cite{moe2013segre} (and similarly for the algorithm of Eklund, Jost and Peterson \cite{Jost} for subschemes of $\pp^n$) is the computation of the saturations to find the residual sets. This can in practice be a quite computationally expensive procedure. Moe and Qviller \cite{moe2013segre} describe their algorithm which uses this result to obtain the Segre classes in Section 5 of \cite{moe2013segre}. Runtime comparisons between the algorithm constructed here and that of \cite{moe2013segre} are given in \S\ref{subsection:PerformenceMultiProj}. 

When computing $s(V,\pp^n)$ for $V$ a subscheme of $\pp^n$ algorithms of Aluffi \cite{aluffi2003computing} and Eklund, Jost, and Peterson \cite{Jost} may also be applied. The algorithm of Aluffi \cite{aluffi2003computing} works by computing the ideal of the blowup of $\pp^n$ along $V$, hence the main computational cost of this algorithm is the cost of computing the Rees algebra. This is often a very computationally expensive operation. When computing $s(V,\pp^n)$ the algorithm presented here reduces to the algorithm of the author in \cite{Helmer2015}, and does indeed offer increased performance in comparison to the algorithms of \cite{aluffi2003computing} and \cite{Jost}, see \cite{Helmer2015} for a detailed comparison in the projective setting.

A separate algorithm to compute $\chi(V)$ and $c_{SM}(V)$ using algebraic methods for $V$ a subscheme of $\pp^n$ was given by Marco-Buzun\'ariz in \cite{BuzEuler2012}. This method, in practice, computes sectional Euler characteristics and its main computational cost arises from the computation of (numerous) primary decompositions, Hilbert polynomials and elimination ideals. As noted by Marco-Buzun\'ariz in \cite[\S8.1]{BuzEuler2012} the computations required are extremely expensive and seem to quickly become impractical even in low dimension and degree.

The main performance benefit of Algorithm \ref{algorithm:SegreAlgMultiProj} seems to be that it explicitly constructs a zero dimensional system, so that we only need to compute the vector space dimension of the quotient rings specified by Theorem \ref{theorem:sat_1mTMultiProj}. While our approach can still be computationally difficult, the explicit nature of the set up allows for a variety of effective algorithms for computing the vector space dimension of a quotient by a zero dimensional ideal to be applied. 

As noted in the introduction the recent algorithm of Harris \cite{harris2015computing} will also allow for the computation of Segre classes of subschemes of a projectively embedded toric variety in the Chow ring of projective space. However in many cases the high codimension of the projective embedding leads to computations in a polynomial ring with significantly more variables (i.e.\ for $\pp^5 \times \pp^5 \times \pp^4$ using the Segre embedding gives a polynomial ring with $6 \cdot 6 \cdot 5=180$ variables, the Cox ring would have $6+6+5=17$ variables).

\subsection{$c_{SM}$ Class of a Hypersurface}\label{subsection:csmHyper}
We now give Theorem I.4 of Aluffi \cite{aluffi1999chern} which will allow us to compute $c_{SM}$ classes by computing Segre classes.  
\begin{propn}[Theorem I.4 of Aluffi \cite{aluffi1999chern}]
Let $V$ be a hypersurface in a nonsingular variety $M$ and let $Y$ be the singularity subscheme of $V$. Then we have \begin{equation}
c_{SM}(V)=c(T_M)\cap \left( s(V,M)+ \sum_{i=0}^n \sum_{j=0}^{n-i} {n-i \choose j}[V]^j \cap (-1)^{n-i}s_{i+j}(Y, M)\right) \label{eq:csm_hyper}
\end{equation} where $[V]$ is the class of $V$ in $A^*(M)$. Here $s_{i+j}(Y, M)$ denotes the dimension $i+j$ component of $s(Y, M)$ and $T_M$ denotes the tangent bundle to $M$. \label{propn:AluffiGeneralCSMSegre_Relation}
\end{propn} 

Consider the case where $M=X_{\Sigma}$ is a smooth complete toric variety and $V=V(f)$ for $f$ a polynomial in the coordinate ring $R=k[x_1,\dots, x_m]$ of $X_{\Sigma}$; that is restricting ourselves to the case considered in this note. By Proposition \ref{propn:SingXJacobianMinors} and Proposition \ref{propn:finPartials} the singularity subscheme $Y$ in the theorem above is the scheme defined by the partial derivatives of $f$ with respect to $x_1, \dots, x_m$. In particular Proposition \ref{propn:finPartials} tells us that we need not include $f$ among the defining equations of $Y$ since it is in the ideal generated by its partial derivatives. %Let $n=\dim(X_{\Sigma})$, let $R=k[x_1,\dots,x_m]$ be the total homogeneous coordinate ring (as in \S\ref{subsection:homhgeneCoordRing}) and let $f(x_1,\dots,x_n)$ be a polynomial in $R$. Similar to the case of $\pp^n$ and its graded coordinate ring the action of the algebraic torus on the coordinates induces a homogeneous property on the polynomials $f\in R$. Let $\lambda=(\lambda_1,\dots, \lambda_m)\in (\CC^*)^m$. By the geometric quotient construction of toric varieties (see Lemma 5.1.1 and Theorem 5.1.11 of Cox, Little and Schenck \cite{david2011toric}) we have that $f(\lambda_1x_1,\dots, \lambda_mx_m)=\lambda_1^{l_1}\cdots\lambda_m^{l_m}f(x),$ for some $l_j\in \ZZ$. From this it follows, in a manner similar to the classic proof of Euler's homogeneous function formula, that $f$ is contained in an ideal generated by its partial derivatives. Hence when $V=V(f)$ is a hypersurface in $X_{\Sigma}$ the singularity subscheme $Y$ in the theorem above is the scheme defined by the partial derivatives of $f$ with respect to $x_1, \dots, x_m$.
 
\section{Main Results}\label{section:mainresultsMultiProj}
In this section we present the main results of this note. Throughout this section and in the following sections we take $X_{\Sigma}$ to be a smooth projective toric variety, unless otherwise stated. This assumption will be somewhat relaxed to $X_{\Sigma}$ a smooth complete toric variety in Theorems \ref{theorem:SegreMultiProj}, \ref{theorem:countingPoints}, and \ref{theorem:csm_complete_int_multi_proj}. The restriction to $X_{\Sigma}$ a smooth projective toric variety is required for the construction of complementary cycles used in Theorem \ref{theorem:sat_1mTMultiProj}. 

In \S\ref{subsection:SegreSubscheme} we prove Theorem \ref{theorem:SegreMultiProj} which extends the result of Proposition 3.1 of Aluffi \cite{aluffi2003computing} to subschemes of smooth complete toric varieties $X_{\Sigma}$. For a subscheme $V$ of $X_{\Sigma}$ this result gives us an expression for the Segre class $s(V,X_{\Sigma})$ in terms of the projective degrees of a rational map $\phi: X_{\Sigma} \dashrightarrow \pp^r $. We then prove Theorem \ref{theorem:sat_1mTMultiProj} which gives an expression for the projective degrees of such a rational map in terms of the dimension of an explicit quotient ring. This theorem is the main ingredient in our algorithms to compute characteristic classes of subschemes of toric varieties. The expression in Theorem \ref{theorem:sat_1mTMultiProj} requires that we have a valid choice of a dehomogenizing ideal so we can consider the relevant intersection in affine space via the quotient construction. 

In Theorem \ref{theorem:countingPoints} we give a simple characterization of a general dehomogenizing ideal which may be used for any zero dimensional subscheme of a toric variety $X_{\Sigma}$ which satisfies what we refer to as the affine codimension condition. While Theorem \ref{theorem:countingPoints} is not needed to apply Theorem \ref{theorem:sat_1mTMultiProj} it does greatly simplify the implementation of a general purpose algorithm and our test implementation used in \S\ref{subsection:PerformenceMultiProj} is restricted to the setting of Theorem \ref{theorem:countingPoints}. 

In Theorem \ref{theorem:csm_complete_int_multi_proj} we give an expression for the $c_{SM}$ class of certain types of complete intersection subschemes of toric varieties of the form $X_{\Sigma}$; this result generalizes Theorem 3.3 of the author \cite{HelmerTCS} and leads to a more efficient algorithm that avoids performing inclusion/exclusion when computing the $c_{SM}$ class in some cases.

\subsection{Counting Points in Zero Dimensional Subschemes} \label{subsection:countPoints}
In this subsection $X_{\Sigma}$ will denote a smooth complete toric variety. The result in this subsection is essentially a consequence of the geometric quotient construction of a toric variety, see Cox, Little and Schenck \cite[\S 5.1,\S 5.2]{david2011toric}. In this subsection we again take $k=\CC$ but note that if $X_{\Sigma}=\pp^{n_1}\times \cdots \times \pp^{n_j}$ we could allow $k$ to be any algebraically closed field of characteristic zero. We also let $\Sigma(1)=\left\lbrace \rho_1, \dots, \rho_m \right\rbrace$, $R$ be the total coordinate ring of $X_{\Sigma}$ with irrelevant ideal $B$, and let $\dim(X_{\Sigma})=n$.  

We briefly review some terminology from \S\ref{section:intro}. We say that $X_{\Sigma}$ satisfies the \textit{affine codimension condition} if the number of primitive collections of the fan $\Sigma$ is equal to $ m-n$ (or equivalently if there are $m-n$ primary components in a primary decomposition of the irrelevant ideal $B$). 

Let $\tilde{R}=k[x_1,\dots,x_m]$ be a polynomial ring in the variables of $R$ but without the grading. Let $W=V(J)$ be a reduced zero dimensional subscheme of $X_{\Sigma}$ consisting of $q$ points. We refer to an ideal $L_A$ in $\tilde{R}$ as a \textit{dehomogenizing ideal} for $W$ if the intersection $V(J)\cap V(L_A) $ in $k^m=\Spec(\tilde{R})$ contains $q$ points. We refer to an ideal $L_A(\Lambda)$ as a \textit{general dehomogenizing ideal} for $X_{\Sigma}$ if for a general choice of scalars (i.e.\ scalars in some Zariski open dense set) the intersection $V(J)\cap V(L_A(\Lambda)) $ in $k^m$ contains $q$ points provided that the set $\Lambda$ is sufficiently general.  

Theorem \ref{theorem:countingPoints} below gives an explicit construction of such a general dehomogenizing ideal provided that $X_{\Sigma}$ satisfies the affine codimension condition. In practice this gives rise to a (theoretical) probability one method which can be used to count points in a zero dimensional subscheme of $X_{\Sigma}$ by instead counting points in affine space. In practical implementations general will be replaced by random, however the probability of failure may be made arbitrarily small; see \cite[Appendix A]{HelmerTCS} for a discussion of this in the projective case. 

Thinking of a projective space $\pp^n$, roughly speaking, the idea behind Theorem \ref{theorem:countingPoints} is to prevent (in general) missing counting points at infinity when dehomogenizing. Rather than thinking of dehomogenization as setting some coordinate equal to one and working in $\mathbb{A}^n$ we think of it as an intersection with $V(\ell)$ in $\mathbb{A}^{n+1}$ for an affine linear form $\ell$ (i.e.\ instead of taking $x_0=1$ we intersect with $V(x_0-1)$). To get a general point at infinity, we then take the linear form $\ell$ to be a general affine linear form. As a toric variety $\pp^n={\rm Proj}(k[x_0,\dots,x_n])$ is defined by a fan with $n+1$ rays. The total coordinate ring is $k[x_0,\dots,x_n]$ and the irrelevant ideal is $B=(x_0,\dots,x_n)$, which is a prime ideal, hence we have only one primitive collection. Since $(n+1)-\dim(\pp^n)=(n+1)-n=1$ is equal to the number of primitive collections then $\pp^n$ satisfies the affine codimension condition. The general dehomogenizing ideal for $\pp^n$ is $
L_A(\Lambda)=(\lambda_0x_0+\cdots+\lambda_nx_n-1)
$ for general $\Lambda=(\lambda_0,\dots,\lambda_n) \in (k^*)^{n+1}$. 

Write the total coordinate ring of $X_{\Sigma}$ as $R=k[x_{\rho_1},\dots,x_{\rho_m}]$ and let $\mathfrak{p}= \left\lbrace \mathbf{p_1},\dots,\mathbf{p_{\nu}} \right\rbrace$ be the set of all (unique) primitive collections of rays in $\Sigma(1)$ (note that each $\mathbf{p_{\ell}}$ is a set of rays in $\Sigma(1)$). In the theorem below we use the fact that all primary components of the irrelevant ideal $B$ of $X_{\Sigma}$ are monomial ideals, i.e.\ from \eqref{eq:irred_decomp_Vanishing_irel_ideal} (see also \cite[Definition 5.15, Proposition 5.1.6]{david2011toric}) we have $$B= \bigcap_{ \left\lbrace \rho_1, \dots ,{\rho_s} \right\rbrace \in \mathfrak{p} } (x_{\rho_1}, \dots, x_{\rho_s}).
$$

\begin{theorem}
 Let $X_{\Sigma}$ be a smooth complete toric variety which satisfies the affine codimension condition. Let $R$, $B$ and $\mathfrak{p}$ be as above and let $V=V(I)$ be any (reduced) dimension zero subscheme of $X_{\Sigma}$. The number of points in $V\subset X_{\Sigma}$ is equal to the number of points in the affine set $V(I) \cap V(L_A(\Lambda)) \subset \mathbb{C}^{m}$ where 
\begin{equation}
L_A(\Lambda)= \left( \sum_{\rho_j \in \mathbf{p_1} }\lambda_{j}^{(1)}x_{\rho_{j}} -1, \dots,  \sum_{\rho_j \in \mathbf{p_{\nu}}}\lambda_{j}^{(\nu)}x_{\rho_{j}}  -1 \right),
\end{equation}for general $\lambda_{j}^{(l)} \in \CC^*$ (where $\CC^*$ denotes the algebraic torus). $\Lambda$ denotes the collection of all $\lambda_{j}^{(l)}$. In other words $L_A(\Lambda)$ is a general dehomogenizing ideal for $X_{\Sigma}$. Note that in the summations above we associate a scalar $\lambda_{j}^{(l)}$ to each monomial in each primitive collection. \label{theorem:countingPoints}
\end{theorem} 

\begin{proof}
Because the $\lambda_{j}^{(l)}$ are general and by our assumption on the number of primitive collections, we have that $V(L_A(\Lambda))$ will have codimension $m-n$ in $\CC^m$.  

From Theorem \ref{theorem:GeoQuoCox} we have a geometric quotient $\pi:\CC^m - Z(B) \to X_{\Sigma} $. Following the terminology of Cox, Little, and Schenck \cite{david2011toric}, given a point $p \in X_{\Sigma}$ we say a point $x \in \pi^{-1}(p)$ gives homogeneous coordinates for $p$. Since $\pi$ is a geometric quotient we have $\pi^{-1}(p)=G \cdot x$. 

Given a homogeneous polynomial $f\in R$ we have that if $f(x)=0$ for one choice of homogeneous coordinates of $p \in X_{\Sigma}$ then $f(x)=0$ for all choices of homogeneous coordinates. Hence to count points in $V$ we may fix a choice of homogeneous coordinates for our points $p_1,\dots, p_l \in X_{\Sigma}$. We may do this as follows. 

For a cone $\sigma\in \Sigma$ let $U_{\sigma}$ denote the affine toric variety of $\sigma$ (see Theorem 1.2.18 of Cox, Little, and Schenck \cite{david2011toric}). Since $X_{\Sigma}$ is smooth and complete, the affine open sets $U_{\sigma}$ for $\sigma$ a maximal cone are a torus invariant affine covering of $X_{\Sigma}$. By Proposition 5.2.10 of Cox, Little, and Schenck \cite{david2011toric} (also see the remark following Proposition 5.2.10 of \cite{david2011toric}) $V \cap U_{\sigma}$, the affine piece of $V$ for each maximal cone $\sigma$, may be obtained by setting $x_{\rho}=1$ for some $\rho\notin \sigma(1)$ in each of the polynomials defining $I$; this gives local coordinates on $X_{\Sigma}$. In our case we may patch this together to give global coordinates by choosing a unique ray $\rho$ from each primitive collection and setting $x_{\rho}=1$ for each of these $\rho$. 

More specifically consider a primitive collection $C$, we know that $C$ is not contained in $\sigma(1)$ for all $\sigma \in \Sigma$ and specifically if we are considering some maximal cone $\sigma$ then $C$ is not in this $\sigma(1)$. Hence there is some ray in $C$ which is not in $\sigma(1)$. Now suppose that we have maximal cones $\sigma_1,\dots, \sigma_j$ and primitive collections $C_1, \dots , C_{m-n}$. We may choose one ray $\rho_i$ from each primitive collection $C_i$ such that at least one of the maximal cones does not contain $\rho_i$; further we know from the structure of the irrelevant ideal that all rays not in a maximal cone will appear in some primitive collection meaning that we may choose appropriate rays from each primitive collection so that we can give compatible local coordinates on each maximal cone. Further, again from the structure of the irrelevant ideal, we see that for each maximal cone $\sigma$ there exists exactly one $\rho$ in each primitive collection which is not in $\sigma$.

Hence by setting $x_{\rho}=1$ for some such $\rho$ in each primitive collection we obtain affine sets $V\cap U_{\sigma}$ for each maximal cone $\sigma$ which cover $V$. Taking the intersection of these sets we must obtain all points in $V$ and we may not obtain points which are not in $V$, hence the intersection of these affine sets must have the same number of points as $V$. 

If we instead take a general linear combination of the $x_{\rho}$ for $\rho \in C_i$ and set this linear combination equal to $1$ for each $C_i$ and work in the larger ambient space $\CC^m$, since the linear combination is general, then the vanishing set of this equation will not contain points which lay in $V(B)$ (since by the construction of $B$, given a point in our homogeneous coordinates for $X_{\Sigma}$, if there is at least one coordinate $x_{\rho}\neq 0$ for some $\rho$ in each primitive collection then this point is not in $V(B)$). Hence by taking such linear combinations we would expect to obtain a new set of affine spaces covering $V$ as a subset of $\CC^m$, provided the linear combination is sufficiently general. Taking the intersection of the new affine covering spaces gives the set $V(I) \cap V(L_A(\Lambda)) \subset \CC^m$, and by the arguments above this space will have dimension zero and hence will consist of points in $\CC^m$. Further, since it is obtained from affine pieces which cover $V$ the number of points in $V(I) \cap V(L_A(\Lambda)) \subset \CC^m$ will be the same as the number of points in $V$. \qedhere
\end{proof}

As a toric variety $\pp^n\times \pp^m={\rm Proj}(k[x_0,\dots,x_n]) \times {\rm Proj}(k[y_0,\dots,y_m])$ is defined by a fan with $(n+1)+(m+1)$ rays. In this case the irrelevant ideal is $B=(x_0,\dots,x_n)\cap (y_0,\dots,y_m),$ so there are two primitive collections. Since $(n+1)+(m+1)-\dim(\pp^n\times\pp^m)=2$ is equal to the number of primitive collections the affine codimension condition is satisfied. The general dehomogenizing ideal for $\pp^n \times \pp^m$ is $
L_A(\Lambda)=(\lambda^{(1)}_0x_0+\cdots+\lambda^{(1)}_nx_n-1, \lambda^{(2)}_0y_0+\cdots+\lambda^{(2)}_my_m-1)
$ for general $\Lambda=(\lambda^{(1)}_0,\dots,\lambda^{(1)}_n, \lambda^{(2)}_0, \dots,\lambda^{(2)}_m ) \in (k^*)^{n+1} \times(k^*)^{m+1} .$

As mentioned in the \S\ref{section:intro}, 42 of the 124 unique smooth Fano fourfolds satisfy the affine codimension condition. We consider one such example below. 
\begin{example}
Let $X_{\Sigma}$ be the toric variety defined by the fan $\Sigma$ with rays $\rho_0=(1, 0, 0, 0)$, $\rho_1=(0, 1, 0, 0)$, $\rho_2=(-1, -1, 0, 0)$, $\rho_3=(1, 0, 1, 0)$, $\rho_4=(0, 0, 0, 1)$, $\rho_5=(0, 0, -1, -1)$ and maximal cones $\left\langle \rho_0,\rho_1,\rho_3,\rho_4 \right\rangle$, $\left\langle \rho_0,\rho_1,\rho_3,\rho_5 \right\rangle$, $\left\langle \rho_0,\rho_1,\rho_4,\rho_5 \right\rangle$, $\left\langle \rho_0,\rho_2,\rho_3,\rho_4 \right\rangle$, $\left\langle \rho_0,\rho_2,\rho_3,\rho_5 \right\rangle$, $\left\langle \rho_0,\rho_2,\rho_4,\rho_5 \right\rangle$, $\left\langle \rho_1,\rho_2,\rho_3,\rho_4 \right\rangle$, $\left\langle \rho_1,\rho_2,\rho_3,\rho_5 \right\rangle$, $\left\langle \rho_1,\rho_2,\rho_4,\rho_5 \right\rangle$.
$X_{\Sigma}$ is a smooth Fano projective toric variety (this is \textsf{smoothFanoToricVariety(4,7)} in the ``NormalToricVarieties" Macaulay2 package \cite{M2}). The total coordinate ring is $R=k[x_0,\dots, x_5]$, ${\rm Pic}(X_{\Sigma}) \cong \ZZ^2$, and the grading on $R$ is given by $$\deg(x_0)=(1,-1),\; \deg(x_1)=\deg(x_2)=(1,0), \; \deg(x_3)=\deg(x_4)=\deg(x_5)=(0,1).$$  The divisors $[V(\rho_1)] =[V(\rho_2)]$ and $[V(\rho_3)]=[V(\rho_4)]=[V(\rho_5)]$ are nef. The fan $\Sigma$ has two primitive collections, $R$ has six variables, and $\dim(X_{\Sigma})=4$. Thus $X_{\Sigma}$ satisfies the affine codimension condition (since $6-4=2$ and there are two primitive collections). The irrelevant ideal of $X_{\Sigma}$ is $B=(x_0,x_1,x_2)\cap (x_3,x_4,x_5),$ so by Theorem \ref{theorem:countingPoints} we have that the general dehomogenizing ideal is $
L_A({\Lambda})=\left( \lambda_0 x_0+\lambda_1 x_1+\lambda_2 x_2-1,   \lambda_3 x_3+\lambda_4 x_4+\lambda_5 x_5-1\right)
$ for general $\Lambda=(\lambda_0,\lambda_1,\lambda_2,\lambda_3,\lambda_4,\lambda_5)\in (\CC^*)^3\times(\CC^*)^3$.
\label{example:smoothFanoToric}
\end{example}

It is worth noting that while the affine codimension condition allows us to employ Theorem \ref{theorem:countingPoints}, which is computationally convenient since it builds a general dehomogenizing ideal, it is not strictly necessary if one simply wishes to dehomogenize the homogeneous coordinates. In particular, for a given subscheme $V$ of some smooth complete toric variety $X_{\Sigma}$, we always have an affine covering of $V$ given by $V\cap U_{\sigma}$ for all maximal cones $\sigma \in \Sigma$ where $U_{\sigma}$ is obtained by setting the coordinate $x_{\rho}=1$ for some $\rho \notin \sigma$, see Proposition 5.2.10 of Cox, Little, and Schenck \cite{david2011toric}. Hence when $V$ is a reduced zero dimensional subscheme we could count the number of points in $V$ by computing in these affine charts and patching together the results. The implementation of any algorithm for computing characteristics classes using the procedure of dehomogenization by affine charts and patching to count points would likely be much more complicated than the algorithm outlines presented in \S\ref{section:SegreCSMMultiProj}. However, given an effective computational procedure for working with the affine charts, the result of Theorem \ref{theorem:sat_1mTMultiProj} could be reformulated and the proof would proceed in a nearly identical manner. Since our method using dehomogenizing ideals allows for a simple and effective implementation we will not explore this alternate approach here, however it would be an interesting topic for further study.

\subsection{The Segre Class of Subschemes}\label{subsection:SegreSubscheme}
In this subsection $X_{\Sigma}$ will denote a smooth complete toric variety. Let $R$ be the graded homogeneous coordinate ring of $X_{\Sigma}$. Let $I$ be an ideal in $R$ which is homogeneous with respect to the grading. Then, since $I$ is homogeneous with respect to the grading, (by Cox \cite[\S3]{cox1995homogeneous}) we may choose generators $I=(f_0,\dots,f_r)$ so that $[V(f_i)]=\alpha$ in $A^1(X_{\Sigma})$ for all $i$. Throughout this subsection we suppose that $\alpha$ is nef. Also let $V=V(I)$ be the closed subscheme of $X_{\Sigma}$ defined by $I$. Observe that $V$ is the base of the linear system defined by the sections $f_0,\dots, f_r$, see Fulton \cite[\S 4.4]{fulton}.

Define a rational map $\phi:X_{\Sigma} \dashrightarrow \pp^r$ given by \begin{equation}
\phi: p \mapsto (f_0(p):\cdots: f_r(p)).\label{eq:rational_map_of_ideal}
\end{equation} Let $
\Gamma_I \subset X_{\Sigma} \times \pp^r$ denote the closure of the graph of $\phi$. Let $\mathfrak{h}$ denote the hyperplane class in $\pp^r$ and let $\pi: \Gamma_I \to X_{\Sigma}$ be the projection. Working from the graph $\Gamma_I$ we define a class 
\begin{equation}
G= \sum_{{\iota}=0}^n [Y_{\iota}]   \in A^*(X_{\Sigma}),  \label{eq:G_biproj}
\end{equation} where $[Y_{\iota}]=\pi_*(\mathfrak{h}^{\iota} \cdot [\Gamma_I])$. Note that by definition $[Y_{\iota}]=\left[ \overline{\phi^{-1}(\pp^{r-{\iota}})} \right]$ where $\pp^{r-{\iota}}$ denotes a general linear subspace of dimension $r-{\iota}$ in $\pp^r$. Put another way $[Y_{\iota}]$ is the class of the closure of the inverse image under $\phi$ of a general linear subspace of codimension ${\iota}$ in $\pp^r$.

\begin{lemma}
With the notations and assumptions above we have that $Y_{\iota}$ is generically reduced, has pure codimension $\iota$, and we have that\begin{equation}
Y_{\iota}= \overline{V(P_1, \dots, P_{\iota})-V(I)}\label{eq:Yiota_class_1}
\end{equation} with the $P_i$ being a general linear combination of $f_0,\dots,f_r$. That is $ P_i=\sum_j \lambda^{(i)}_j f_j$ for general $\lambda^{(i)}_j \in k $ prescribed by the general linear subspace $\pp^{r-{\iota}}$. Further, the class $[Y_{\iota}]$ will be the same for any general choice of $\lambda^{(i)}_j \in k $. \label{lemma:Yiota}
\end{lemma}\begin{proof}By construction $V$ forms the base of a linear system (see for example Fulton \cite[\S 4.4]{fulton}), hence $\phi$ is a morphism away from $V$. Let $\tilde{\phi}:X_{\Sigma}-V\to \pp^r$ be this morphism. To compute $[Y_{\iota}]$ we must compute the class of the inverse image of a general linear subspace $\pp^{r-{\iota}}$ under the morphism $\tilde{\phi}$. By Kleiman's transversality theorem \cite{kleiman1974transversality} (see also Eisenbud and Harris \cite[Theorem 1.7 (b)]{eisenbud20163264}) we have that $Y_{\iota}=\tilde{\phi}^{-1}(\pp^{r-{\iota}})$ is generically reduced and has pure codimension $\iota$. Rather than working with the map $\tilde{\phi}$ we may instead work directly with $\phi$ provided we avoid $V$, since these maps agree outside of $V$. The coordinates of the image of the rational map $\phi$ are given by $f_0,\dots, f_r$ thus, away from $V$, $\tilde{\phi}^{-1}(\pp^{r-{\iota}})$ is given by the intersection of $\iota$ hypersurfaces defined by general linear combinations of $f_0,\dots, f_r$; this gives \eqref{eq:Yiota_class_1}. Note that while $Y_{\iota}$ will depend on the choice of the linear subspace $\pp^{r-{\iota}}$ (which corresponds to a choice of $\lambda^{(i)}_j \in k $) the class $[Y_{\iota}]$ will not as long as the choice is sufficiently general, i.e.\ provided the scalars lie in the Zariski dense set prescribed by Kleiman's transversality theorem. \qedhere
\end{proof}

 Also note that $[Y_{\iota}]=\alpha^{\iota}$ for ${\iota}<\codim(V)$ since $V$ has no components of codimension less than $\codim(V)$, i.e.\ for  ${\iota}<\codim(V) $ \begin{equation}
[Y_{\iota}]= \left[ {V(P_1+\cdots+P_{\iota})}\right]. \label{eq:Yiota_iotaLessThanCodimV}
\end{equation}Further note that \begin{equation}
[Y_{\iota}]=0 \;\;\; \mathrm{for} \;\iota>r. \label{eq:Yiota_gre_min_n_r}
\end{equation}

Take $\omega_1^{(\iota)},\dots \omega_m^{(\iota)}$ to be a basis for $A^{\iota}(X_{\Sigma})$, then the class $[Y_{\iota}] \in A^*(X_{\Sigma})$ will have the form \begin{equation}
[Y_{\iota}]=\sum_{i=1}^m \gamma_{i}^{(\iota)}\omega_i^{(\iota)} ,\label{eq:projectiveMultiDegrees}
\end{equation}we refer to the $\gamma_{i}^{(\iota)}$ as the \textit{projective degrees} of the rational map $\phi$. Note that these projective degrees reduce to the usual projective degrees when $X_{\Sigma}=\pp^n$ is a single projective space (see \cite[Example 19.4]{harris1992algebraic}). We will, however, often find it notationally simpler to work with the classes $[Y_{\iota}]$ and the class $G$ of \eqref{eq:G_biproj}.
 
We now state a notation of Aluffi \cite[\S 1.4]{aluffi1995singular} for operations in the Chow ring of a smooth algebraic variety $M$. Let $\alpha = \sum_{i\geq 0}\alpha^{(i)}$ be a cycle class in $A^*(M)$ with $\alpha^{(i)}$ denoting the piece of $\alpha$ of codimension $i$ in $A^*(M)$, that is $\alpha^{(i)} \in A^i(M)$. Also take $\mathcal{L}$ to be some line bundle on $M$. In this setting define the following notation (which will be used in the proof of the theorem below):
 \begin{equation}
  \alpha \otimes \mathcal{L} = \sum_{i \geq 0} \frac{\alpha^{(i)}}{c(\mathcal{L})^i}. \label{eq:ALuffi_chow_ring_tensor_notations}
\end{equation}

\begin{theorem} 
Let $X_{\Sigma}$ be a smooth complete toric variety with total coordinate ring $R$. Let $V=V(I)$ be a subscheme of $X_{\Sigma}$ defined by an ideal $I=(f_0,\dots, f_r)$ homogeneous with respect to the grading on $R$ and assume, without loss of generality, that $[V(f_i)]=\alpha$ in $ A^1(X_{\Sigma})$ for all $i$. Further suppose that $\alpha$ is nef. Then we have that
$$\vspace{-2mm}
 s(V,X_{\Sigma}) =1- \frac{1}{c(\oo(\alpha))}\cdot \sum_{\iota \geq 0} \frac{[Y_{\iota}]}{(1+\alpha)^{\iota} }.
$$ 

\label{theorem:SegreMultiProj}
\end{theorem} \begin{proof}
By construction, the graph $\Gamma_I$ is isomorphic to the blow-up of $X_{\Sigma}$ along $V$, $Bl_V(X_{\Sigma})$. Note that $V$ is the zero scheme of a section of $\oo(\alpha)^{r+1}$. Let $E=\pi^{-1}(V)$ be the exceptional divisor of the blow-up $Bl_V(X_{\Sigma})$. From Fulton \cite[\S4.4]{fulton} (which we may apply since $\alpha$ is nef and hence generated by global sections) we have that $ \sigma^*(\oo_{\pp^r}(1))=\pi^*(\oo_{X_{\Sigma}}(\alpha)) \otimes \oo(-E)$ where $\sigma:Bl_V(X_{\Sigma}) \to \pp^r$ is the projection; let $[E]$ be the class of the exceptional divisor in the Chow ring of $X_{\Sigma}\times \pp^r$. From this we have $\mathfrak{h}=\alpha-[E]$ and hence $[E]=\alpha-\mathfrak{h}$. Applying Fulton \cite[Corollary 4.2.2]{fulton} (given in \eqref{eq:general_SegreDef} above) we have $$
s(V,X_{\Sigma})= \pi_* \left( \frac{[E]}{1+[E]}\right)=\pi_* \left( \frac{\alpha- \mathfrak{h}}{1+\alpha-\mathfrak{h}}\right).
$$
We may simplify this expression as follows:\footnotesize \begin{align*}
\pi_* \left( \frac{\alpha-\mathfrak{h}}{1+\alpha-\mathfrak{h}}\right) =&\; \pi_* \left( \frac{[\Gamma_I](1+\alpha-\mathfrak{h}) -[\Gamma_I]}{1+\alpha-\mathfrak{h}}\right) \\
=&\;\pi_* \left([\Gamma_I]- \frac{1}{1+\alpha-\mathfrak{h}}\cdot [\Gamma_I]\right) \\
=&\;1-\pi_* \left(\frac{1}{1+\alpha} \cdot \frac{1+\alpha}{1+\alpha-\mathfrak{h}}\cdot [\Gamma_I]\right) \\
=&\;1-\frac{1}{c(\oo(\alpha))} \cdot\pi_* \left(\left( \frac{1}{1-\frac{\mathfrak{h}}{1+\alpha}}\cdot [\Gamma_I] \right)  \right)\\
=&\;1-\frac{1}{c(\oo(\alpha))} \cdot\pi_* \left(\left( \frac{1}{1-\mathfrak{h}}\cdot [\Gamma_I] \right) \otimes \oo(\alpha) \right)\\
=&\;1-\frac{ G \otimes \oo(\alpha) }{{c(\oo(\alpha))}}.
\end{align*}\normalsize This concludes the proof.\qedhere \end{proof}

We remark that Theorem \ref{theorem:SegreMultiProj} generalizes the result of Aluffi \cite[Proposition 3.1]{aluffi2003computing}. 

\subsection{Computing the Projective Degrees}

We now state and prove a result which will allow us to compute the classes $[Y_{\iota}]$ of \eqref{eq:G_biproj}, and hence to compute the Segre class via Theorem \ref{theorem:SegreMultiProj}, using a computer algebra system by calculating the projective degrees $\gamma_{i}$ as in \eqref{eq:projectiveMultiDegrees}.

In this subsection we take $X_{\Sigma}$ to be a smooth projective toric variety with $\dim(X_{\Sigma})=n$; the restriction to $X_{\Sigma}$ a projective (rather than a complete) smooth toric variety is needed for the construction of the complementary cycles used below. As above we let $R=k[x_1,\dots, x_m]$ be the graded total coordinate ring of $X_{\Sigma}$ and let $I$ be an ideal in $R$ which is homogeneous with respect to the grading. Since $X_{\Sigma}$ is a smooth projective toric variety then the codimension one Chow group $A^1(X_{\Sigma})$ will have a basis consisting of nef divisors, see Proposition 6.3.24 of of Cox, Little and Schenck \cite{david2011toric}. As in \S\ref{subsection:SegreSubscheme} we choose generators $I=(f_0,\dots,f_r)$ so that $[V(f_i)]=\alpha \in A^1(X_{\Sigma})$ for all $i$.  Additionally, since $A^1(X_{\Sigma})$ has a basis of nef divisors, we may also assume $\alpha$ is nef.

Take $[Y_{\iota}]  = \sum_{ i=1}^\mu \gamma_{i}^{(\iota)} \omega_i^{(\iota)} $ with the $\gamma_{i}^{(\iota)}\in \ZZ$ being the projective degrees of the rational map $\phi: X_{\Sigma} \dashrightarrow \pp^r$ specified by $\phi:p \mapsto (f_0(p):\cdots:f_r(p))$ and with $ \omega_1^{(\iota)}, \dots, \omega_\mu^{(\iota)}$ a basis of $A^{\iota}(X_{\Sigma})$. Let $b_1,\dots, b_{q} \in A^1(X_{\Sigma})$ denote a fixed nef basis for $A^1(X_{\Sigma})$. Since the divisors $b_j$ are nef we may express the rational equivalence class of a point as a monomial in $b_1,\dots, b_{q}$. In particular let $\zeta=b_1^{c_1}\cdots b_{q}^{c_q}$ denote the rational equivalence class of a point in the dimension zero Chow group, $A_0(X_{\Sigma})$. Similarly we may write the basis elements $ \omega_1^{(\iota)}, \dots, \omega_\mu^{(\iota)}\in A^{\iota}(X_{\Sigma})$ as monomials in $b_1,\dots, b_{q}$. Since the $b_j$ are nef and since $\zeta$ is the class of a point then each exponent of $b_j$ appearing in $\omega_i^{(\iota)}$ must be less or equal to $c_j$, the exponent of $b_j$ in $\zeta$. Hence $\zeta$ is divisible by $\omega_i^{(\iota)}$. Recall from \S\ref{section:intro} that we refer to the class $a_i^{(\iota)}=\zeta/\omega_i^{(\iota)}$ as the \textit{complementary cycle} to $\omega_i^{(\iota)}$. For $b\in A^1(X_{\Sigma})$ let $\ell(b)$ be a general form in $R$ with $ [\ell(b)]=b \in A^*(X_{\Sigma})$. Writing $a_i^{(\iota)}=b_1^{j_1}\cdots b_q^{j_q}$ for $b_1,\dots, b_{q} \in A^1(X_{\Sigma})$ let $L_{a_i^{(\iota)}}$ be the ideal generated by $ j_1$ forms $\ell(b_1)$, $ j_2$ forms $\ell(b_2)$,$\dots$, and $ j_q$ forms $\ell(b_q)$. We refer to the ideal $L_{a_i^{(\iota)}}$ as the \textit{complementary ideal} associated to $\omega_i^{(\iota)}$.

The statement of the theorem below makes use of a dehomogenizing ideal $L_A$ of a (reduced) dimension zero subscheme of $X_{\Sigma}$. If we restrict to toric varieties $X_{\Sigma}$ satisfying the affine codimension condition of \S\ref{subsection:countPoints} we may let $L_A$ be the general dehomogenizing ideal $L_A(\Lambda)$ (for general $\Lambda$) of Theorem \ref{theorem:countingPoints}. 

%As noted in \S\ref{section:intro}, if we restrict to $X_{\Sigma}$ a smooth projective toric variety then the requirement in Theorem \ref{theorem:sat_1mTMultiProj} for $b_1,\dots, b_{q}$ to be nef is redundant since . 

 \begin{theorem} Let $X_{\Sigma}$ be a smooth projective toric variety with $R$ its total homogeneous coordinate ring. Let $I=(f_0,\dots,f_r)$ be an ideal in $R$ as above, whigh generators choosen so that $[V(f_i)]=\alpha \in A^1(X_{\Sigma})$ for all $i$ and so that $\alpha$ is nef. With the notation above we have that the projective degrees of the rational map specified by $\phi:p \mapsto (f_0(p):\cdots:f_r(p))$ are given by \begin{equation}
 \gamma_{i}^{(\iota)}=\dim_k \left( R[T]/ \left(( P_1, \dots, P_{\iota},S)+L_{a_i^{(\iota)}} +L_A \right) \right), 
\label{eq:mainProjDegreFormula}
 \end{equation}
 where the $P_{\ell}$ are general linear combinations of $f_0,\dots, f_r$, $L_{a_i^{(\iota)}}$ is the complementary ideal to $\omega_i^{(\iota)}\in A^{\iota}(X_{\Sigma})$, $L_A$ is a dehomogenizing ideal of $V( P_1, \dots, P_{\iota})\cap V\left( L_{a_i^{(\iota)}} \right)$ and $
S= 1-T\sum_{l=0}^r \vartheta_l f_l$ for general $\vartheta_l \in k$. Further we have that 
\begin{align*}
&[Y_{\iota}]=\alpha^{\iota} \in A^*(X_{\Sigma}) \;\;\; \mathrm{for\;} \iota=0,\dots, \codim(V)-1, \; \mathrm{and} \\
&[Y_{\iota}]=0 \in A^*(X_{\Sigma}) \;\;\; \mathrm{for\;} \iota>\min(n,r).
\end{align*}
\label{theorem:sat_1mTMultiProj} 
\vspace{-8mm}
\end{theorem}
\begin{proof}
The statement for $\iota< \codim(V)$ is given in \eqref{eq:Yiota_iotaLessThanCodimV} and the statement for $\iota>\min(n,r)$ is given in \eqref{eq:Yiota_gre_min_n_r}. Before proving \eqref{eq:mainProjDegreFormula}, we give a brief overview. The goal is to compute the projective degree $ \gamma_{i}^{(\iota)}$ by constructing an appropriate zero dimensional ideal and applying algebraic methods.% Note that $\gamma_i^{(\iota)}$ is the coefficient of $\omega_i^{(\iota)}$ in the class $[Y_{\iota}]$; intersection theory (think in particular a `moving lemma', see Eisenbud and Harris \cite[Appendix A]{eisenbud20163264}) tells us that this coefficient corresponds to the number of points in an intersection of $Y_{\iota}$ with a sufficiently general variety of class $a_i^{(\iota)}$. 

By appropriate choices of general forms in $R$ we may construct the zero dimensional ideal corresponding to the class $[Y_{\iota}]\cdot a_i^{(\iota)}$. Lemma \ref{lemma:Yiota} gives us an explicit expression for $Y_{\iota}$, in light of this we need to consider an ideal defined by general linear combinations of the generators of $I$, but we must ensure no points in $V(I)$ are included. To accomplish this we use an argument in the style of the Rabinowitsch trick; that is we add a variable $T$ and an equation which cannot be satisfied by points in $V(I)$ given a general choice of some scalars. The structure of this added equation necessitates working in the affine algebraic setting. This will be accomplished by adding the dehomogenizing ideal $L_A$ (here we employ the geometric quotient construction of $X_{\Sigma}$, see Theorem \ref{theorem:GeoQuoCox}). It is important to remember that while the varieties defined in the proof may depend on a choice of forms or constants the number of points in the resulting reduced dimension zero scheme will not (provided the choice is general). 

Following this outline our argument will have three main steps. The first step is the construction of a subscheme of $X_{\Sigma}$ whose closure corresponds to the intersection product $[Y_{\iota}]\cdot a_i^{(\iota)}$. The second step is writing this as a closed subscheme of $X_{\Sigma} \times \mathbb{A}^1$. The final step is to dehomogenize via the dehomogenizing ideal $L_A$ and obtain an intersection in affine space. 

\textbf{Step 1} (writing an intersection corresponding to $[Y_{\iota}]\cdot a_i^{(\iota)}$). Take $\iota$ such that $\codim(V) \leq \iota \leq \min(n,r)$. We wish to compute the class $[Y_{\iota}]$ in the Chow ring $A^*\left( X_{\Sigma}\right)$; by Lemma \ref{lemma:Yiota} $Y_{\iota}$ is the closure of the open set $
\tilde{Y_{\iota}}=V(P_1, \dots, P_{\iota})-V(I).
$ From the definition of $a_i^{(\iota)}$ we have that $\omega_j^{(\iota)} a_i^{(\iota)} =0$ for $i\neq j$ and, hence, that $
[Y_{\iota}]\cdot a_i^{(\iota)} = \gamma_i^{(\iota)} \zeta
$. Since all relevant line bundles are generated by global sections (due to our nef assumption) it follows from Kleiman's transversality theorem (see Kleiman \cite[Theorem 2]{kleiman1974transversality} or Eisenbud and Harris \cite[Theorem 1.7]{eisenbud20163264}) that there exist Zariski open dense sets $U_1,U_2$ so that for constants $\lambda_{l,j}$ and linear forms $\ell{(b_j)}$ chosen in $U_1$ and $U_2$ respectively we have that $$\tilde{Y_{\iota}}\cap V(L_{a_i^{(\iota)}})=\bigcap_{l=1}^{{\iota}}V \left(\sum_{j=0}^{r} \lambda_{l,j} f_j\right) \cap V(L_{a_i^{(\iota)}}) -V(f_0,\dots ,f_r)$$ is smooth (scheme-theoretically), reduced, and has dimension $0$. In other words, the zero dimensional set associated to the class $\gamma_i^{(\iota)} \zeta $ is given by $\tilde{Y_{\iota}}\cap V(L_{a_i^{(\iota)}})$, meaning to find $ \gamma_i^{(\iota)} $ we must find the degree of $\tilde{Y_{\iota}}\cap V(L_{a_i^{(\iota)}})$, i.e.\ the number of points in $\tilde{Y_{\iota}}\cap V({L_{a_i^{(\iota)}}})$. Hence we wish to compute $$ \gamma_i^{(\iota)} =\mathrm{card} \left( \bigcap_{l=1}^{{\iota}}V \left(\sum_{j=0}^{r} \lambda_{l,j} f_j\right) \cap V(L_{a_i^{(\iota)}}) -V(f_0,\dots ,f_r)\right),$$ where $\mathrm{card}$ denotes the number of points in a zero dimensional set.

%Now if we choose sufficiently general linear forms (so that all intersections are transverse, which is possible since all relevant line bundles are generated by global sections due to our nef assumption) then the zero dimensional set associated to $\gamma_i^{(\iota)} \zeta $ is given by $$
%\tilde{Y_{\iota}}\cap V(L_{a_i^{(\iota)}})=\left( V( P_1, \dots, P_{\iota})-V(I) \right) \cap V(L_{a_i^{(\iota)}})
%$$ and hence to find $ \gamma_i^{(\iota)} $ we must find the degree of $\tilde{Y_{\iota}}\cap V(L_{a_i^{(\iota)}})$, i.e.\ the number of points in $\tilde{Y_{\iota}}\cap V({L_{a_i^{(\iota)}}})$. Hence we wish to compute $$ \gamma_i^{(\iota)} =\mathrm{card} \left( \bigcap_{l=1}^{{\iota}}V \left(\sum_{j=0}^{r} \lambda_{l,j} f_j\right) \cap V(L_{a_i^{(\iota)}}) -V(f_0,\dots ,f_r)\right),$$ where $\mathrm{card}$ denotes the number of points in a zero dimensional set. 
%By Kleiman's transversality theorem (see Kleiman \cite[Theorem 2]{kleiman1974transversality} or Eisenbud and Harris \cite[Theorem 1.7]{eisenbud20163264}) we have that there exists Zariski open dense sets $U_1,U_2$ so that for constants $\lambda_{l,j}$ and linear forms $\ell{(b_j)}$ chosen in $U_1$ and $U_2$ respectively we have that $$\widetilde{W}=\bigcap_{l=1}^{{\iota}}V \left(\sum_{j=0}^{r} \lambda_{l,j} f_j\right) \cap V(L_{a_i^{(\iota)}}) -V(f_0,\dots ,f_r)$$ is smooth (scheme-theoretically), reduced, and has dimension $0$. 

It is important at this point to note that, for our purposes here, since we just wish to compute projective degrees we are only interested in the \textit{number} of points in $\tilde{Y_{\iota}}\cap V(L_{a_i^{(\iota)}})$ and not in the points themselves. Hence while the scheme $\tilde{Y_{\iota}}\cap V(L_{a_i^{(\iota)}})$ does depend on the choice of $\lambda_{l,j}$ and $\ell{(b_j)}$ the number of points in $\tilde{Y_{\iota}}\cap V(L_{a_i^{(\iota)}})$ does not, provided we choose $\lambda_{l,j}$ and $\ell{(b_j)}$ from the Zariski dense sets $U_1$ and $U_2$ prescribed by Kleiman's transversality theorem. Hence we obtain the desired projective degree for any general choice of $\lambda_{l,j}$ and $\ell{(b_j)}$.

\textbf{Step 2} (moving to $X_{\Sigma}\times \mathbb{A}^1$). Let $
W=\bigcap_{l=1}^{{\iota}}V \left(\sum_{j=0}^{r} \lambda_{l,j} f_j\right) \cap V(L_{a_i^{(\iota)}}).
$ In what follows we fix $\lambda_{l,j}$ and $\ell{(b_j)}$ so that they lay in the desired sets $U_1$ and $U_2$.  Hence we may write the set $\tilde{Y_{\iota}}\cap V(L_{a_i^{(\iota)}})$ as a finite collection of points, that is we may write $W -V(f_0,\dots ,f_r)=\left\lbrace p_0,\dots, p_s\right\rbrace$. Then $$U_3=\pp^{r} - \bigcup_{i=0}^sV \left(f_0(p_i)y_0+\cdots +f_r(p_i)y_r \right)$$ is open and dense in $\pp^{r}=\Proj(k[y_0,\dots,y_r])$, because $(f_0(p_i),\dots,f_r(p_i))\neq (0,\dots,0)$ for all $i$. Take $ \vartheta=(\vartheta_0,\dots ,\vartheta_r)\in U_3$; then $$\left( W\cap V \left( \sum_{j=0}^r \vartheta_{j} f_j \right) \right) -V(f_0,\dots ,f_r)$$ is empty. Now consider the ideals $L_{a_i^{(\iota)}}$ and $ \left(\sum_{j=0}^r \lambda_{l,j} f_j \right)$ as ideals in the ring $ R[T]$, and define $V_S=V(S)$ where $ S= 1-T \cdot \sum_{j=0}^r \vartheta_{j} f_j \in R[T].$ 

For a point $p\in V(f_0,\dots,f_r)$ we have that $f_j(p)=0,$ for $j=0,1,\dots, r$ which implies that $p$ is not in $V_S$ since $p$ cannot be a solution to the equation $ 1-T \cdot \sum_{j=0}^r \vartheta_{j} f_j=0$. Now take $p \in W-V(f_0,\dots ,f_r)$ then $$T_p=\frac{1}{\sum_{j=0}^r \vartheta_{j} f_j(p)}$$ is well defined since for $\vartheta\in U_3$ we have that $W\cap V \left( \sum_{j=0}^r \vartheta_{j} f_j \right)  -V(f_0,\dots ,f_r) $ is empty, so $(p,T_p) \in V_S$. Now let $\widehat{W} \subset  X_{\Sigma} \times \mathbb{A}^1$ be the variety given by a linear embedding of $W$ in $ X_{\Sigma} \times \mathbb{A}^1$, where $\mathbb{A}^1=\Spec(k[T])$. We have $
\pi(\widehat{W}\cap V_S)=W-V(f_0,\dots,f_r),
$ where $\pi $ is the projection $\pi:  X_{\Sigma}  \times \mathbb{A}^1 \to  X_{\Sigma} $, and in particular $
\mathrm{card}(\widehat{W}\cap V_S)=\mathrm{card}(W-V(f_0,\dots,f_r)).
$

\textbf{Step 3} (dehomogenizing $X_{\Sigma}$). By assumption $L_A$ is chosen to be a dehomogenizing ideal (if $X_{\Sigma}$ satisfies the affine codimension condition we could take $L_A=L_A(\Lambda)$ to be the general dehomogenizing ideal prescribed by Theorem \ref{theorem:countingPoints}), hence rather than considering the intersection $\widehat{W}\cap V_S $ in $ X_{\Sigma} \times \mathbb{A}^1$ we may consider $W$ to be a set of reduced points in $ \mathbb{A}^{m}$ via the geometric quotient construction of $X_{\Sigma}$ (see Theorem \ref{theorem:GeoQuoCox}). That is we intersect with the vanishing of the dehomogenizing ideal $V(L_A)$ in $\mathbb{A}^m$ which gives $$W= \bigcap_{\ell=0}^{\iota}V \left(\sum_{j=0}^{r} \lambda_{\ell,j} f_j\right) \cap V({L_{a_i^{(\iota)}}}) \cap V(L_A)\subset \mathbb{A}^{m}$$ and we then consider the intersection $\widehat{W}\cap V_S$ in $\mathbb{A}^{m+1}$. As the points in $\widetilde{W}$ are reduced the cardinality of the zero dimensional set $$ \bigcap_{\ell=0}^{{\iota}}V \left(\sum_{j=0}^{r} \lambda_{\ell,j} f_j\right) \cap V(L_{a_i^{(\iota)}}) \cap V(L_A) \cap V_S \subset \mathbb{A}^{m+1}$$ is given by the vector space dimension of $
 R[T]/(( P_1, \dots, P_{\iota},S)+L_{a_i^{(\iota)}}+L_A).
$\qedhere\end{proof}
\begin{example}
Let $X_{\Sigma}$ be the smooth Fano fourfold of Example \ref{example:smoothFanoToric}. Recall that the total coordinate ring is $R=k[x_0,\dots, x_5]$, with $\ZZ^2\cong {\rm Pic}(X_{\Sigma})$ grading given by $\deg(x_0)=(1,-1),$ $\deg(x_1)=\deg(x_2)=(1,0),$ $\deg(x_3)=\deg(x_4)=\deg(x_5)=(0,1).$ Using the isomorphism of Proposition \ref{propn:ChowRingDef} we have that $$A^*(X_{\Sigma})\cong \ZZ[h_0,\dots,h_5]/(h_0h_1h_2,h_3h_4h_5,h_0-h_2+h_3,h_1-h_2,h_3-h_5,h_4-h_5)$$ with $h_i=[V(\rho_i)]$ where $\Sigma(1)=\left\lbrace\rho_0,\dots,\rho_5 \right\rbrace$ are the generating rays. The divisors $h_2$ and $h_5$ are nef and form a basis for $A^1(X_{\Sigma})\cong {\rm Pic}(X_{\Sigma})$, we write classes in $A^*(X_{\Sigma})$ in terms of this basis. Consider the ideal in $R$ given by $I=(x_2^3x_3x_4^9-15x_2^3x_3^5x_5^5,5x_1^2x_2x_3^5x_4^5+x_1^2x_2x_3x_5^9)$ and the corresponding subscheme $V=V(I)$ of $X_{\Sigma}$. $V$ is a singular subscheme of $X_{\Sigma}$ and $\codim(V)=1$. Using Theorem \ref{theorem:sat_1mTMultiProj} and the general dehomogenizing ideal of Example \ref{example:smoothFanoToric} we compute the projective degrees of the associated rational map and obtain the class $G$ (as in \eqref{eq:G_biproj}) to be $$
G=\sum_{\iota}[Y_{\iota}]=1+2h_2+10h_5 \in A^*(X_{\Sigma}).
$$ For this $V$ the class $\alpha$ in Theorem \ref{theorem:SegreMultiProj} is given by $ \alpha=3h_2+10h_5$, applying Theorem \ref{theorem:SegreMultiProj} gives $$
s(V,X_{\Sigma})=16083h_2^2h_5^2-414h_2^2h_5-1680h_2h_5^2+3h_2^2+46h_2h_5+120h_5^2+h_2-h_5 \in A^*(X_{\Sigma}).
$$
\label{example:SegreSmoothFano}
\end{example}
\subsection{The $c_{SM}$ Class of Complete Intersections}\label{subsection:CSM_Complete_int}
In this subsection we prove Theorem \ref{theorem:csm_complete_int_multi_proj} which extends the result of the author \cite[Theorem 3.3]{HelmerTCS} to the case where $V=V(I)$ is a subscheme of a smooth complete toric variety  $ X_{\Sigma}$. While this theorem can be seen as a repacking of the result of Proposition \ref{propn:AluffiGeneralCSMSegre_Relation} (this is discussed in the proof of Theorem \ref{theorem:csm_complete_int_multi_proj} below) it is nonetheless a helpful result for computational purposes. In particular this result allows us to avoid the inclusion/exclusion procedure of Proposition \ref{propn:csm_higher_codim} in certain cases. 

Inclusion/exclusion has two significant computational costs. First, given an ideal $I$ with $r$ generators, we must compute $2^r$ Segre classes of the corresponding singularity subschemes. Second, since we must consider unions, the generators of the ideals of the singularity subschemes considered may have substantially higher degree than the degree of the generators of $I$. Thus it is desirable to avoid the inclusion/exclusion procedure whenever possible. The result of Theorem \ref{theorem:csm_complete_int_multi_proj} gives us a method to do this in some cases. Using this we construct Algorithm \ref{algorithm:csm_Complete_Int_multi_proj}. 
\begin{theorem} Let $ X_{\Sigma}$ be a dimension $n$ smooth complete toric variety and let $V=V(f_0,\dots,f_r)$ be a possibly singular global complete intersection subscheme of $ X_{\Sigma}$ which can be written as the intersection of a smooth subvariety and an hypersurface in $X_{\Sigma}$. Let the hypersurfaces $$V_0=V(f_0),\dots,V_r=V(f_r)$$ be ordered such that we have that $V_0\cap \cdots \cap  V_{r-1}$ is smooth. In $A^*(X_{\Sigma})\cong \ZZ[x_1,\dots,x_m]/(\mathcal{I}+\mathcal{J})$ (here we write $A^*(X_{\Sigma})$ in the notation of Proposition \ref{propn:ChowRingDef}) we have \scriptsize \begin{equation}
c_{SM}(V)= \frac{(1+x_1)\cdots (1+x_m) }{(1+[V_0])\cdots(1+[V_r])} \cdot \left([V_0]\cdots [V_r]+ \left( (-1)^r\sum_{j=0}^r\sum_{i=0}^j {r-i \choose j-i} (-1)^i [V_r]^{j-i}c_i \right) \cdot \left( \sum_{i=0}^n \frac{(-1)^is^{(i)}(Y,X_{\Sigma})}{(1+[V_r])^i}\right) \right), \label{eq:csm_complete_int_multi_proj}
\end{equation} \normalsize where $c_i$ is the dimension $i$ component of $(1+[V_0])\cdots(1+[V_r])$,  $s^{(i)}(Y,X_{\Sigma})$ is the codimension $i$ component of the Segre class of $Y$ in $X_{\Sigma}$ where $Y$ denotes the singularity subscheme of $V$ and $x_i=[V(\rho_i)] \in A^*(X_{\Sigma})$ where $\left\lbrace \rho_1,\dots, \rho_m \right\rbrace =\Sigma(1)$ are the generating rays. 
\label{theorem:csm_complete_int_multi_proj}
\end{theorem}
\begin{proof}
By assumption $ V(f_r)$ is a hypersurface in the smooth variety $V_0\cap \cdots \cap  V_{r-1}$, this allows us to apply Proposition \ref{propn:AluffiGeneralCSMSegre_Relation}. A quick simplification of the resulting expression in $A^*(X_{\Sigma})$ is given by employing Theorem 1.1 of Fullwood \cite{fullwood2014milnor}. In our case the result of \cite[Theorem 1.1]{fullwood2014milnor} gives the following,
\begin{equation}
c_{SM}(V)=c(T_{X_{\Sigma}})\cap s(V,X_{\Sigma})+\frac{(-1)^r c(T_{X_{\Sigma}})}{c(\mathcal{E}) } \cdot \left( c( \mathcal{E}^{\vee} \otimes \mathcal{L}) \cdot \left( \sum_{i=0}^n \frac{(-1)^is^{(i)}(Y,X_{\Sigma})}{(1+[V_r])^i}\right) \right)
\end{equation}
where $\mathcal{E}$ is the line bundle associated to $V_0 \cap \cdots \cap V_{r}$ and $\mathcal{L}$ is the line bundle associated to $V_r $. Hence $c(\mathcal{E})=(1+[V_0])\cdots(1+[V_r])$ and $c(\mathcal{L})=1+[V_r]$. Note that we have $c(T_{X_{\Sigma}})=(1+x_1)\cdots (1+x_m) $, see Cox \cite[Proposition 13.1.2]{david2011toric}. Since $V$ is a complete intersection we have that $$
s(V,X_{\Sigma})=\frac{[V_0]\cdots [V_r]}{(1+[V_0])\cdots(1+[V_r])}. 
$$  Using Remark 3.2.3 of Fulton \cite{fulton} to expand $ c( \mathcal{E}^{\vee} \otimes \mathcal{L})$ gives the expression in \eqref{eq:csm_complete_int_multi_proj}. 
\qedhere \end{proof}

If a complete intersection subscheme $V$ of $X_{\Sigma}$ does not satisfy the assumptions of Theorem \ref{theorem:csm_complete_int_multi_proj} the theorem may still be applied in conjunction with a specialized form of inclusion/exclusion to reduce the total number of inclusion/exclusion steps required. We state a specialized form of the inclusion/exclusion property that is helpful in such cases in Proposition \ref{propn:Inclusion_exclusion_singular_part_only} below; this result follows directly from the inclusion/exclusion property of the $c_{SM}$ class.

\begin{propn}
Let $X$ be a smooth variety. Let $Z\subset X_{\Sigma}$ be smooth (scheme-theoretically) and let $ V_1=V(f_1),$ $V_2=V(f_2)$ be singular hypersurfaces in $X_{\Sigma}$. If $V=Z\cap V_1\cap V_2$, then we have \begin{equation}
c_{SM}(V)=c_{SM}(Z\cap V_1) + c_{SM}(Z\cap V_2)-c_{SM}(Z\cap (V_1 \cup V_2)) \in A^*(X), \label{eq:IE_to_extend_on_sing_gen_CSM}
\end{equation}
 here $V_1 \cup V_2$ is the scheme generated by $f_1 \cdot f_2$.  \label{propn:Inclusion_exclusion_singular_part_only}
\end{propn} 
When $Z$ is a smooth complete intersection subscheme of $X_{\Sigma}$ each of the terms in (\ref{eq:IE_to_extend_on_sing_gen_CSM}) can be computed using Theorem \ref{theorem:csm_complete_int_multi_proj}. This gives an algorithm which reduces the number of inclusion/exclusion steps required when working with subschemes which are the intersection of a smooth complete intersection subscheme and some collection of hypersurfaces. 
\section{Algorithms}\label{section:SegreCSMMultiProj}
In this section we summarize how the results of \S\ref{section:mainresultsMultiProj} can be used to construct algorithms to compute characteristic classes of subschemes of smooth complete toric varieties. To simplify the presentation of the algorithms in this section we restrict to the case where $X_{\Sigma}$ is a smooth projective toric variety satisfying the affine codimension condition of \S\ref{subsection:countPoints}.

%This restriction to $X_{\Sigma}$ projective could be relaxed, however more restrictive assumptions on $V$ would then be required. In particular we would need to be able to write all complementary cycles associated to $V$ (see the definitions preceding Theorem \ref{theorem:sat_1mTMultiProj}) as a product of nef divisors so that Theorem \ref{theorem:sat_1mTMultiProj} could be applied to compute all projective degrees appearing in \eqref{eq:projectiveMultiDegrees} for all $[Y_{\iota}]$ appearing in Theorem \ref{theorem:SegreMultiProj} for the given $V$.

It should be recalled, as discussed in \S\ref{section:intro}, that all algorithms presented here are probabilistic; this is because they involve a general choice of some scalars which will in practice be replaced by a random choice. Based on experimental results, making random choices from a subset of the coefficient field with more than $25000$ elements seems to result in a failure rate below 1 in 2000. The failure rate can be further decreased by choosing from larger subsets.

Let $n=\dim(X_{\Sigma})$ and let $R=k[x_{\rho_1},\dots, x_{\rho_m}]$ be the total coordinate ring of $X_{\Sigma}$. The class of a point will be denoted $\zeta \in A_0(X_{\Sigma}) $, and $ V=V(I)$ will be a subscheme of $X_{\Sigma}$ defined by an ideal $I=(f_0,\dots, f_r)$ homogeneous with respect to the grading on $R$. The generating rays of the fan $\Sigma$ will be denoted $\left\lbrace \rho_1,\dots, \rho_m \right\rbrace =\Sigma(1)$. All algorithms will represent the Chow ring $A^*(X_{\Sigma})$ via the isomorphism of Proposition \ref{propn:ChowRingDef}, that is $
A^*(X_{\Sigma})\cong \ZZ[x_1,\dots,x_m]/(\mathcal{I}+\mathcal{J})
$ for $\mathcal{I}$ and $\mathcal{J} $ as in Proposition \ref{propn:ChowRingDef}. Let $B$ be the irrelevant ideal of $X_{\Sigma}$ and let $\mathfrak{p}= \left\lbrace \mathbf{p_1},\dots,\mathbf{p_{\nu}} \right\rbrace$ be the set of all (unique) primitive collections of rays (each $\mathbf{p_{\ell}}$ is a set of rays in $\Sigma(1)$) so that $$B= \bigcap_{ \left\lbrace \rho_1, \dots ,{\rho_s} \right\rbrace \in \mathfrak{p} } (x_{\rho_1}, \dots, x_{\rho_s}).
$$For $\beta\in A^1(X_{\Sigma})$ let $R.\mathrm{random}(\beta)$ be a function which creates a general polynomial $f$ in $R$ such that $[V(f)]=\beta\in A^*(X_{\Sigma})$.

We first present Algorithm \ref{algorithm:SegreAlgMultiProj}, an algorithm to compute the Segre class of a subscheme of $X_{\Sigma}$ via the results of Theorems \ref{theorem:countingPoints}, \ref{theorem:SegreMultiProj}, and \ref{theorem:sat_1mTMultiProj}. In Algorithm \ref{algorithm:SegreAlgMultiProj} we will assume, without loss of generality, that the generators of $I$ are chosen so that $[V(f_j)]=\alpha$ for all $j$. Further suppose that $\alpha$ is nef. 

\begin{algorithm} 
 \textbf{Input:} A smooth projective toric variety $X_{\Sigma}$ and an ideal $I=(f_0,\dots, f_r)$ defining a subscheme $V=V(I)$ of $X_{\Sigma}$ as above.\newline 
 \textbf{Output:} $s(V,X_{\Sigma} )$ in $A^*(X_{\Sigma} )$. \vspace{-4mm}
\begin{itemize} 
\item Let $ P_{j}=\sum_{l=0}^r  \lambda_{j,l} f_l \;$ for $j=1,\dots,n$ and for general $ \lambda_{j,l}$.

\item \textbf{For ${\iota}=\codim(V)$ to $ \min(n,r)$:}
\begin{itemize}
\item $J_{\iota}=R[T].\mathrm{ideal}(P_1,\dots,P_{\iota})$.
\item $K_{\iota}=J_{\iota}+R[T].\mathrm{ideal}\left(1-T\cdot \sum_{j=0}^r \vartheta_j f_j \right)$; $\vartheta_j$ a general scalar in ${k}$.
\item Let  $\Omega^{({\iota})}= \left\lbrace \omega_1^{({\iota})}, \dots, \omega_{\nu}^{({\iota})} \right\rbrace $ denote the monomial basis of $A^{\iota}(X_{\Sigma})$. 
\item \textbf{For $\omega$ in $\Omega^{({\iota})}$:} \begin{itemize}
\item Let $ a^{(\iota)}=\frac{\zeta}{\omega}$ and factor $a^{(\iota)}=b_1^{j_1}\cdots b_q^{j_q}$ for $b_j\in A^1(X_{\Sigma})$.
\item $ L= \sum_{w=0}^{j_1} R[T].\mathrm{ideal}(R.\mathrm{random}(b_1) )+ \cdots + \sum_{w=0}^{j_q} R[T].\mathrm{ideal}(R.\mathrm{random}(b_q) )$
\item $
L_A= R[T].\mathrm{ideal}\left(   \sum_{\rho_j \in \mathbf{p_1} }\lambda_{j}^{(1)}x_{\rho_{j}} -1, \dots,  \sum_{\rho_j \in \mathbf{p_{\nu}}}\lambda_{j}^{(\nu)}x_{\rho_{j}}  -1  \right)$, for general $\lambda_{j}^{(\ell)}$.
\item Set $ \gamma_{\omega}=\dim_k \left(R[T]/\left( K_{\iota}+ L+L_A\right) \right)$.
\end{itemize}
\item Set $\displaystyle [Y_{\iota}]=\sum_{\omega \in \Omega^{({\iota})}} \gamma_{\omega} \cdot \omega \in A^*(X_{\Sigma}).$
\end{itemize}
\item\textbf{Return } $\displaystyle {s(V,X_{\Sigma} )= 1- \frac{1}{1+\alpha}\cdot \sum_{\iota \geq 0} \frac{[Y_{\iota}]}{(1+\alpha)^{\iota} } \in  A^*(X_{\Sigma})}$.
%\item \textbf{Return $s(V,X_{\Sigma})$}.
\end{itemize}
 \label{algorithm:SegreAlgMultiProj}
\end{algorithm} \vspace{-5mm}
\normalsize

In Algorithm \ref{algorithm:csm_InExMultiProj} we give an algorithm which uses Proposition \ref{propn:AluffiGeneralCSMSegre_Relation}, the inclusion/exclusion property of $c_{SM}$ classes, and Algorithm \ref{algorithm:SegreAlgMultiProj} to compute $c_{SM}(V)$ for $V$ a subscheme of $X_{\Sigma}$. 
\begin{algorithm} 
 \textbf{Input:} A smooth projective toric variety $X_{\Sigma}$ and an ideal $I=(f_0,\dots, f_r)$ defining a subscheme $V=V(I)$ of $X_{\Sigma}$ as above.\newline
 \textbf{Output:} $c_{SM}(V)$ in $A^*(X_{\Sigma})$ and/or $\chi(V)$. \vspace{-4mm}
\begin{itemize}
\item Let $\mathrm{csm}=0 \in A$.
\item Let $\mathcal{S}$ be the set of all distinct non-empty subsets of $\left\lbrace f_0,\dots, f_r \right\rbrace $.
\item \textbf{For $\left\lbrace f_{i_1},\dots,f_{i_{\mathfrak{s}}} \right\rbrace\in \mathcal{S}$} \begin{itemize}
\item Let $g=f_{i_1}\cdots f_{i_{\mathfrak{s}}} $ in $R$.
\item Let $J$ be the Jacobian ideal of $g$, that is the ideal defining the singularity subscheme $Y=V(J)$ of $W=V(g)$. $J$ is generated by the partial derivatives of $g$ (see \S\ref{subsection:csmHyper}). 
\item Let $[W]=[V(g)]$.
\item Calculate $s(W,X_{\Sigma})=s(V(g),X_{\Sigma})=\frac{[W]}{1+[W]} \in A^*(X_{\Sigma}).$
\item Compute $s(Y,X_{\Sigma})=s(V(J),X_{\Sigma})\in A^*(X_{\Sigma})$ using Algorithm \ref{algorithm:SegreAlgMultiProj}.
\item $c(T_{X_{\Sigma}})=(1+x_1)\cdots (1+x_m)\in A^*(X_{\Sigma}) $.\small
\item $\displaystyle{\mathrm{csm}=\mathrm{csm}+(-1)^{\mathfrak{s}+1} c(T_{X_{\Sigma}})\cdot \left(s(W,X_{\Sigma}) +\sum_{j=0}^n \sum_{l=0}^{n-j} {n-j \choose l}[W]^l \cdot (-1)^{n-j}s_{j+l}(Y, X_{\Sigma})  \right).}$\normalsize
\end{itemize}
\item $c_{SM}(V)=\mathrm{csm}$, set $\chi(V)$ equal to the coefficient of $\zeta$ in $c_{SM}(V)$.
\item \textbf{Return $c_{SM}(V)$ and/or $\chi(V)$}
\end{itemize}
\label{algorithm:csm_InExMultiProj}
\end{algorithm}\vspace{-5mm}

Algorithm \ref{algorithm:csm_Complete_Int_multi_proj} directly computes the $c_{SM}$ class of a complete intersection subscheme which can be written as the intersection of an embedded smooth variety and a hypersurface. Avoiding inclusion/exclusion is often more efficient, since less Segre class computations are performed. Let $V$ be a subscheme of $X_{\Sigma}$ as above. In the algorithm below we use Proposition \ref{propn:SingXJacobianMinors} to write the equations defining the singularity subscheme of $V$. 

\begin{algorithm} 
 \textbf{Input:} A smooth projective toric variety $X_{\Sigma}$ and an ideal $I=(f_0,\dots, f_r)$ as above where $V=V(I)$ is a complete intersection subscheme and $V(f_0)\cap \cdots \cap  V(f_{r-1})$ is smooth. \newline 
 \textbf{Output:} $c_{SM}(V)$ in $A^*(X_{\Sigma})$ and/or $\chi(V)$. \vspace{-4mm}
\begin{itemize}
\item Let $K$ be the ideal defined by the $(r+1)\times (r+1)$ minors of the Jacobian matrix of $I$.
\item Let $J=(K+I):B^{\infty}$  so that $Y=V(J)$ is the singularity subscheme of $V$.
\item Compute $s(Y,X_{\Sigma})\in A^*(X_{\Sigma})$ using Algorithm \ref{algorithm:SegreAlgMultiProj}.
\item Set $c_i$ equal to the dimension $i$ part of $(1+[V(f_0)])\cdots(1+[V(f_r)])\in A^*(X_{\Sigma})$.
\item $s(V,X_{\Sigma})=\frac{(1+x_1)\cdots (1+x_m) }{(1+[V(f_0)])\cdots(1+[V(f_r)])}$.
\item \scriptsize $ \displaystyle{
c_{SM}(V)= s(V,X_{\Sigma}) \cdot \left([V(f_0)]\cdots [V(f_r)]+ \left( (-1)^r\sum_{j=0}^r\sum_{i=0}^j {r-i \choose j-i} (-1)^i [V(f_r)]^{j-i}c_i \right) \cdot \left( \sum_{i=0}^n \frac{(-1)^is^{(i)}(Y,X_{\Sigma})}{(1+[V(f_r)])^i}\right) \right) } $ \normalsize
\item Set $\chi(V)$ equal to the coefficient of $\zeta$ in $c_{SM}(V)$.
\item \textbf{Return $c_{SM}(V)$ and/or $\chi(V)$}
\end{itemize}
\label{algorithm:csm_Complete_Int_multi_proj}
\end{algorithm}\vspace{-5mm}
Note that, via Proposition \ref{propn:Inclusion_exclusion_singular_part_only}, Algorithm \ref{algorithm:csm_Complete_Int_multi_proj} can be extended to work on any subscheme of a smooth complete intersection in $X_{\Sigma}$. This allows us to reduced the number of inclusion/exclusion steps and to increase the efficiency of $c_{SM}$ class computations in some cases. 
\section{Performance} \label{subsection:PerformenceMultiProj}
As above $X_{\Sigma}$ will be denote a smooth complete toric variety. In \S\ref{subsection:runtimeTests} we discuss the real life performance of our algorithms to compute Segre classes, $c_{SM}$ classes and the Euler characteristic of subschemes of $X_{\Sigma}$. Runtime bounds for Algorithm \ref{algorithm:SegreAlgMultiProj} and Algorithm \ref{algorithm:csm_InExMultiProj} are given in \S\ref{subsection:RunTimeBoundMultiProj}. 

\subsection{Runtime Tests}\label{subsection:runtimeTests}

In Table \ref{table:SegreResultsMultiProj} we compare the runtimes of Algorithm \ref{algorithm:SegreAlgMultiProj} to the runtimes of the algorithm of Moe and Qviller \cite{moe2013segre}. The implementation of Moe and Qviller linked to in \cite{moe2013segre} is used for testing. 

All test computations were performed in Macaulay2 \cite{M2} (version 1.8) over $\mathbb{GF}(32749)$ on a computer with a 2.9GHz Intel Core i7-3520M CPU and 8 GB of RAM. A list of the ideals defining the examples presented in this subsection may be found at \url{https://github.com/Martin-Helmer/char-class-calc-toric/blob/master/Examples.m2}, note that all examples are singular and hence could not be considered general in any particular sense. It seems in practice that the runtime of Algorithm \ref{algorithm:SegreAlgMultiProj} is primarily influenced by the degree and number of generators of the ideal, the number of variables in the total coordinate ring, the codimension, and the sparsity of the polynomials defining the ideal (i.e.\ the number of monomials in each polynomial); this is consistent with the result of Proposition \ref{propn:run_time_b_Segre_Multi_proj} below. 

As can be seen in Table \ref{table:SegreResultsMultiProj} Algorithm \ref{algorithm:SegreAlgMultiProj} is consistently and often quite considerably faster than the algorithm of Moe and Qviller \cite{moe2013segre}. The main computational cost of the algorithm of Moe and Qviller \cite{moe2013segre} is the computation of the saturations to find the residual sets. This can, in practice, be a quite computationally expensive procedure. The main computational cost of Algorithm \ref{algorithm:SegreAlgMultiProj} is the computation of the projective degrees via the result of Theorem \ref{theorem:sat_1mTMultiProj}. The computational performance advantage of Algorithm \ref{algorithm:SegreAlgMultiProj} seems to be primarily due to the extremely explicit nature of the result of Theorem \ref{theorem:sat_1mTMultiProj}. This result, in particular, allows us to take advantage of a variety of fast algorithms to solve zero dimensional systems and gives more control over how computations are performed in our implementation. Even in the case where $X_{\Sigma}=\pp^n$, where the algorithm of \cite{moe2013segre} reduces to that of Eklund, Jost and Peterson \cite{Jost}, Algorithm \ref{algorithm:SegreAlgMultiProj} still offers markedly improved performance in comparison with either the implementation of \cite{Jost} or the implementation of \cite{moe2013segre}, see \cite{Helmer2015} for a more on the $X_{\Sigma}=\pp^n$ case.

As discussed in \S\ref{section:intro} and \S\ref{subsection:MoeQviller} the algorithm of Harris \cite{harris2015computing} could be applied to compute the examples in Table \ref{table:SegreResultsMultiProj}. Using the implementation of Harris \cite{harris2015computing} (linked to in \cite{harris2015computing}) we attempted to compute all the examples in Table \ref{table:SegreResultsMultiProj} which are subschemes of a product of projective spaces via the Segre embedding. None of the examples finished computing in 600 seconds. For reference, the runtime of the algorithm of \cite{harris2015computing} to compute the Segre class $s(V(f),\pp^2\times \pp^2)$ for $f$ a general form of degree $(3,3)$ was approximately $135$ seconds, using Algorithm \ref{algorithm:SegreAlgMultiProj} this computation took less than $0.1$ seconds. 

\begin{table}[h!]
\centering
\resizebox{.7\linewidth}{!}{
\begin{tabular}{@{} l *5c @{}}
\toprule 
 \multicolumn{1}{c}{{\color{Ftitle} \textbf{Input}}}    & {\color{Ftitle} toricSegreClass (\cite{moe2013segre}) }  &   {\color{Ftitle} Algorithm \ref{algorithm:SegreAlgMultiProj} }  \\ 
 \midrule 
  Codim.\ $3$ in $\pp^2 \times \pp^3$  & - & {\color{line} 33.6s}\\ 
  Codim.\ $2$ in $\pp^1 \times \pp^1 \times \pp^1$  & 32.0s & {\color{line} 0.1s}\\ 
  Codim.\ $2$ in $\pp^3 \times \pp^2$  & 2.0s & {\color{line} 0.2s}\\
  Hypersurface in $\pp^5 \times \pp^3$  & 147.4s & {\color{line} 0.5s}\\ 
  Codim.\ $2$ in $\pp^2 \times \pp^3 \times \pp^1$  & 66.8s & {\color{line} 0.5s}\\ 
  Codim.\ $2$ in $\pp^2 \times \pp^2 \times \pp^2$  & 15.7s & {\color{line} 0.5s}\\
  Codim.\ $2$ in $\pp^4 \times\pp^3 \times \pp^3$  & - & {\color{line} 7.4s}\\ 
  Codim.\ $2$ in $\pp^4 \times \pp^3 \times \pp^5$  & - & {\color{line} 22.4s}\\ 
  Codim.\ $4$ in $\pp^2 \times \pp^2 \times \pp^1$  & - & {\color{line} 2.7s}\\
 Codim.\ $1$ with 2 gens. in Dim. 3 $X_{\Sigma_1}$  & 7.6s & {\color{line} 0.1s}\\
 Codim.\ $1$ with 3 gens. in Dim. 3 $X_{\Sigma_1}$  & - & {\color{line} 1.0s}\\
  Example \ref{example:SegreSmoothFano}  & 0.6s & {\color{line} 0.1s}\\
\bottomrule
 \end{tabular}}
 \caption[Comparision of algorithms to compute the Segre class of a subscheme of a smooth complete toric variety]{Runtimes of different algorithms for computing the Segre class of a subscheme of a some $X_{\Sigma}$.  The - denotes computations that were stopped after 600 seconds. \label{table:SegreResultsMultiProj}}
 \end{table}
 
In Table \ref{table:CSMResultsMultiProj} we give the running times to compute the $c_{SM}$ class and/or Euler characteristic using Algorithm \ref{algorithm:csm_InExMultiProj}. There are no other implemented algorithms to compute the $c_{SM}$ class and Euler characteristic in this setting, so we are not able to offer direct comparisons. Any algorithm to compute Segre classes $s(W,X_{\Sigma})$ for $W$ a subscheme of $X_{\Sigma}$ could be adapted to compute $c_{SM}$ classes of subschemes of $X_{\Sigma}$ using inclusion/exclusion (Proposition \ref{propn:csm_higher_codim}). The speed of all such computations would depend on the speed of the required Segre class computations, subsequently it seems likely that Algorithm \ref{algorithm:csm_InExMultiProj} would offer a performance advantage.    

\begin{table}[h!]
\centering
\resizebox{.7\linewidth}{!}{
\begin{tabular}{@{} l *4c @{}}
\toprule 
 \multicolumn{1}{c}{{\color{Ftitle} \textbf{Input}}}  & {\color{Ftitle} Algorithm \ref{algorithm:csm_InExMultiProj} }  \\ 
 \midrule 
  Codim.\ $2$ in $\pp^2 \times \pp^2$  & {\color{line} 0.3s}\\
  Codim.\ $2$ in $\pp^6 \times \pp^2$ with deg.\ $(3,0), (0,2)$ eqs.  & {\color{line} 3.9s}\\
  Codim.\ $2$ in $\pp^5 \times \pp^3$ with deg.\ $(2,1)$ and $(1,1)$ eqs.  & {\color{line} 12.4s}\\
  Codim.\ $2$ in $\pp^2 \times \pp^2 \times \pp^3$ with deg.\ $(2,1,0)$ and $(0,1,2)$ eqs.  & {\color{line} 4.8s} \\
  Codim.\ $3$ in $\pp^2 \times \pp^2 \times \pp^3$ with deg.\ $(2,1,0),(0,1,2),(1,2,0)$ eqs.  \quad \quad  & {\color{line} 52.4s} \\
  Codim.\ $1$ in with 2 gens. Dim. 3 $X_{\Sigma_1}$  & {\color{line} 0.3s}\\
  Codim.\ $1$ in with 3 gens. Dim. 3 $X_{\Sigma_1}$  & {\color{line} 2.0s}\\
\bottomrule
 \end{tabular}}
 \caption{Running time of Algorithm \ref{algorithm:csm_InExMultiProj} to find $c_{SM}(V)$ and $\chi(V)$ for subschemes $V$ of $X_{\Sigma}$.  \label{table:CSMResultsMultiProj}}
\end{table}

In Table \ref{table:CSMCompleteIntResultsMultiProj} we compare the running times of Algorithm  \ref{algorithm:csm_Complete_Int_multi_proj}, our direct algorithm to compute the $c_{SM}$ class and Euler characteristic using Theorem \ref{theorem:csm_complete_int_multi_proj}, to the running time of Algorithm \ref{algorithm:csm_InExMultiProj}, our algorithm using inclusion/exclusion in $X_{\Sigma}$. The runtime of Algorithm \ref{algorithm:csm_Complete_Int_multi_proj} includes the time required to compute the singularity subscheme, which is often a considerable percentage of the overall runtime of the algorithm. As such, a more efficient way to compute the singularity subscheme than that presented in Algorithm \ref{algorithm:csm_Complete_Int_multi_proj} could result in a more marked performance gain versus inclusion/exclusion. 
\begin{center}

\begin{table}[h!]
\centering
\resizebox{.6\linewidth}{1.25cm}{
\begin{tabular}{@{} l *5c @{}}
\toprule 
 \multicolumn{1}{c}{{\color{Ftitle} \textbf{Input}}} & {\color{Ftitle} Algorithm \ref{algorithm:csm_InExMultiProj} }  & {\color{Ftitle} Algorithm \ref{algorithm:csm_Complete_Int_multi_proj} }  \\ 
 \midrule 
  Codimension $3$ in $\pp^2 \times \pp^2$  &1.6s & {\color{line} 0.3s}\\
  Codimension $2$ in $\pp^2 \times \pp^3$  &1.9s & {\color{line} 1.0s}\\
  Codimension $3$ in $\pp^2 \times \pp^2 \times \pp^2$  &5.7s & {\color{line} 0.2s}\\
  Codimension $2$ in $\pp^3 \times \pp^2 \times \pp^2$  &3.1s & {\color{line} 0.9s}\\
  
\bottomrule
 \end{tabular}}
 \caption[The $c_{SM}$ class and Euler characteristic of certain compmplete intersections]{Times to compute $c_{SM}(V)$ and $\chi(V)$ for subschemes $V$ of $X_{\Sigma}$ satisfying the assumptions of Theorem \ref{theorem:csm_complete_int_multi_proj} using Algorithm \ref{algorithm:csm_InExMultiProj} and Algorithm \ref{algorithm:csm_Complete_Int_multi_proj}. \label{table:CSMCompleteIntResultsMultiProj}}
 \end{table}
 
 \end{center}
 
\subsection{Running Time Bounds}\label{subsection:RunTimeBoundMultiProj}
Here we consider running time bounds for Algorithms \ref{algorithm:SegreAlgMultiProj} and \ref{algorithm:csm_InExMultiProj}. In the case where $X_{\Sigma}=\pp^n$ complexity results of a different flavour for the problem of computing projective degrees of a rational map as in \eqref{eq:rational_map_of_ideal} and the problem of computing the Euler characteristic of a possibly singular subvariety of $\pp^n$ can be found in B{\"u}rgisser, Cucker, and Lotz \cite{burgisser2005counting}. Roughly speaking \cite[Theorem 1.2]{burgisser2005counting} this tells us that these problems are difficult among the set of problems involving counting points in zero dimensional algebraic sets constructed via generic choices. This result of \cite{burgisser2005counting} is consistent with Proposition \ref{propn:run_time_b_Segre_Multi_proj} and Corollary \ref{corr:run_time_b_csm_multi_proj}.

Throughout this subsection let $\delta(D,N)$ be the number of arithmetic operations required to find the number of points in a zero dimensional affine variety $W$ defined by a polynomial system containing $N$ degree $D$ polynomials in $N$ variables.

Using the algorithm of Lecerf \cite{Lecerf03} we have that the number of arithmetic operations to solve such a system is polynomial in $\oo(N^5D^{3N})$. There also exist bounds of similar order on some Gr\"{o}bner basis algorithms for solving zero dimensional systems. See, for example, Hashemi and Lazard \cite{HashemiLazardGbBounds} or Faug\`{e}re, Gianni, Lazard, and Mora \cite{faugere1993efficient}. Note that while this bound is essentially polynomial in the B\'ezout bound $D^N$, which is the upper bound on the actual number of solutions, $S$, the complexity is still exponential relative to the number of digits, $\log(S)$, in a computer representation of the number $S$. In practice factors like the sparsity of the terms in the defining polynomials will also play a role in the cost $\delta(D,N)$.

In the results below (as in \S\ref{section:SegreCSMMultiProj}) we let $X_{\Sigma}$ be a smooth projective toric variety satisfying the affine codimension condition of \S\ref{subsection:countPoints}. Let $\dim(X_{\Sigma})=n$, let $R$ be the total coordinate ring, and let $N$ be the number of generating rays in $\Sigma(1)$. Take $I=(f_0,\dots,f_r)$ to be a homogeneous ideal, with respect to the grading, in $R$ and let $V=V(I)$ be the subscheme defined by $I$. Further assume, without loss of generality, that $\deg(f_i)=\alpha \in A^1(X_{\Sigma})$ for all $i=0,\dots,r$ and let $D$ be the sum of the exponents of the monomial in $\alpha$ having the largest total degree.

\begin{propn}
Let $X_{\Sigma}$, $I$, $N$, $D$ and $\delta$ be as above and suppose that $\alpha$ is nef. We have that the number of arithmetic operations required to compute the Segre class $s(V,X_{\Sigma} )$ using Algorithm \ref{algorithm:SegreAlgMultiProj} is of order$$
\oo \left(\delta(D+1,N+1) \cdot \sum_{\iota=\codim(V)}^{\min(n,r)} \left( \sum_{i=\iota}^n (-1)^{i-\iota} {i \choose \iota} \left| \Sigma(\iota)\right| \right)  \right),
$$ where $\left| \Sigma(\iota)\right|$ denotes the number of cones in $\Sigma$ of dimension $\iota$. \label{propn:run_time_b_Segre_Multi_proj}
\end{propn}\begin{proof} 
By Danilov \cite[Theorem 10.8]{danilov1978geometry} a basis of the Chow group $A^{\iota}(X_{\Sigma})$ will contain $$
\mathrm{rank} \left(A^{\iota}(X_{\Sigma}) \right)=\sum_{i=\iota}^n (-1)^{i-\iota} {i \choose \iota} \left| \Sigma(\iota)\right|, 
$$elements. For each element we must solve one linear system in an affine space of dimension $N+1$. The largest total degree of a polynomial appearing in the systems we consider will be one plus the sum of the exponents of the monomial in $\alpha$ having the largest total degree.\qedhere\end{proof}
Examining Algorithm \ref{algorithm:csm_InExMultiProj} we note that one Segre class, namely that of the appropriate singularity subscheme, must be calculated for each subset of the generators of $I$ when finding $c_{SM}(V(I))$.

\begin{corr}
Let $X_{\Sigma}$, $I$, $N$, $D$ and $\delta$ be as above and suppose that $\alpha$ is nef. Let $\kappa$ be the minimum codimension of the singularity subscheme of all hypersurfaces of all products of the generators of $I$. The number of arithmetic operations required to compute $c_{SM}(V)$ using Algorithm \ref{algorithm:csm_InExMultiProj} has order $$\oo \left(2^{r+1} \cdot \delta((r+1)\cdot D+1,N+1) \cdot \sum_{\iota=\kappa}^{n} \left( \sum_{i=\iota}^n (-1)^{i-\iota} {i \choose \iota} \left| \Sigma(\iota)\right| \right)   \right).$$ \label{corr:run_time_b_csm_multi_proj}
\end{corr}
\begin{proof}
There are $2^{r+1}$ subsets of $\left\lbrace f_0,\dots,f_r \right\rbrace$. The maximum total degree of elements in the Jacobian ideal of $ f_0\cdots f_r$ will be $(r+1)\cdot D$.\qedhere \end{proof}
\begin{small}

\noindent
{\bf Acknowledgements.} The author was supported by a Natural Sciences and Engineering Research Council of Canada (NSERC) postdoctoral fellowship during the preparation of this note. The author would like to thank the anonymous referees for their many helpful comments and suggestions and would also like to thank \'{E}ric Schost for helpful comments on earlier drafts of this note. 
\end{small}

\begin{appendices}
\section{The Singularity Subscheme in Cox Coordinates}
Let $X_{\Sigma}$ be a smooth complete toric variety with graded total coordinate ring (i.e.~Cox ring) $R$ and let $V$ be a subscheme of $X_{\Sigma}$. In this appendix we prove two results regarding the equations in $R$ defining $V_{\rm Sing}$, the singularity subscheme of $V$ in $X_{\Sigma}$. These results are fairly direct consequences of the structure of the graded total coordinate ring, the geometric quotient construction of $X_{\Sigma}$ and the toric ideal-variety correspondence, see Chapter 5 of Cox, Little and Schenck \cite{david2011toric}. It is likely that both propositions are well known, but we could not find a precise reference so we give short proofs here. We use these results in the algorithms presented in \S\ref{section:SegreCSMMultiProj}.
\begin{propn}
Let $X_{\Sigma}$ be a smooth complete toric variety with graded total coordinate ring $R=k[x_1,\dots, x_m]$. Let $f\in R$. Then we have that the polynomial $f$ is contained in the ideal in $R$ generated by the partial derivatives of $f$, that is we have that $
f \in \left( \frac{df}{dx_1},\dots,  \frac{df}{dx_m} \right).
$\label{propn:finPartials}
\end{propn}
\begin{proof}
Let $G={\rm Hom}_{\ZZ}({\rm Cl}(X_\Sigma), \CC^*)$ be the subgroup of the algebraic torus $(\CC^*)^m$ specified by Lemma 5.1.1 of Cox, Little and Schenck \cite{david2011toric} (see also \S\ref{subsection:homhgeneCoordRing} above). By \cite[Theorem 5.1.11, \S5.2]{david2011toric} we have that $R$ is a multi-graded ring and its elements are homogeneous in the sense that, if we fix $f(x)\in R$ and $g=(\lambda_1,\dots, \lambda_m)\in G$, we have:$$
f(g\cdot x)=f(\lambda_1\cdot x_1,\dots, \lambda_m\cdot x_m)=\lambda_1^{l_1}\cdots\lambda_m^{l_m}f(x),
$$where $l_j\in \ZZ$. Note that we may have $\lambda_{j_1}=\cdots=\lambda_{j_v}$ for some $j_1,\dots, j_v$ determined by $G$. Differentiating with respect to some $\lambda_{j_1}=\cdots=\lambda_{j_v}$ using the chain rule we obtain \footnotesize $$
x_{j_1}\frac{df}{x_{j_1}}(\lambda_1\cdot x_1,\dots, \lambda_m\cdot x_m)+\cdots + x_{j_v}\frac{df}{x_{j_1}}(\lambda_1\cdot x_1,\dots, \lambda_m\cdot x_m)=l_{j_1}\cdots l_{j_v} \lambda_1^{l_1}\cdots\lambda_{j_1}^{l_{j_1}-1}\cdots \lambda_{j_v}^{l_{j_v}-1} \cdots \lambda_m^{l_m}f(x).
$$\normalsize Setting $\lambda_1= \cdots= \lambda_m=1$ in the above shows $f$ is in the ideal defined by the partial derivatives. This proves the proposition. Note that in the case where $X_{\Sigma}=\pp^n$ we have $G=\left\lbrace (\lambda, \dots, \lambda)\in (\CC^*)^m \; | \; \lambda\in \CC^*\right\rbrace\cong \CC^*.$ This gives $\lambda=\lambda_1=\cdots = \lambda_m\in \CC^*$ with the $l_j$'s being the power to which $x_j$ appears and their sum being the degree of the homogeneous polynomial $f$. From this we get the classical Euler's homogeneous function formula (and its corollaries) for singly graded homogeneous polynomials. 
\end{proof}

\begin{propn}
Let $X_{\Sigma}$ be a smooth complete toric variety with total homogeneous coordinate ring $R=k[x_1,\dots,x_m]$ and with irrelevant ideal $B$ so that $X_{\Sigma}= (\CC^m-V(B))/G$. Let $V=V(f_1,\dots, f_s)$ be a subscheme of $X_{\Sigma}$ defined by the $B$-saturated ideal $I=(f_1,\dots, f_s)$ (i.e.~$I=I:B^{\infty}$) in $R$. Let $$
J(I)=\begin{pmatrix}
\frac{df_1}{dx_1} & \cdots & \frac{df_1}{dx_m}\\
\vdots & \ddots & \vdots\\
\frac{df_s}{dx_1} & \cdots & \frac{df_s}{dx_m}
\end{pmatrix}
$$ be the Jacobian matrix. Set $n=\dim(X_\Sigma)$ and $r=\dim(V)$. Define $K$ be the ideal in $R$ generated by all $(n-r)\times (n-r)$ minors of the Jacobian matrix $J(I)$. The singularity subscheme of $V$ is given by $V_{\rm Sing}=V((I+K):B^{\infty})$.\label{propn:SingXJacobianMinors}
\end{propn}
\begin{proof}
Since the ideal $I$ is $B$-saturated and $X_{\Sigma}$ is smooth then, by the toric ideal-variety correspondence (see \cite[Proposition 5.2.7]{david2011toric}), $I$ is the unique ideal associated to $V$. Consider the $r+(m-n)$ dimensional subscheme $\tilde{V}$ of $\CC^m$ defined by the ideal $I$ in the affine space $\CC^m$; that is let $\tilde{V}$ be the affine cone over $V$. Since $\tilde{V}$ is defined by the ideal $I$, a point $p\in \tilde{V}$ is a singular point if and only if the Jacobian matrix $J(I)$ has less than maximal rank at $p$, equivalently $p$ is a singular point if and only if all $(m-(r+(m-n)))\times (m-(r+(m-n)))=(n-r)\times (n-r)$ minors of $J(I)$ vanish at $p$. Hence the singularity subscheme of $\tilde{V}$ in $\CC^m$ is defined by the ideal $I+K$ where $K$ is generated by the $(n-r)\times (n-r)$ minors of $J(I)$. By the geometric quotient construction of $X_{\Sigma}$ (see Theorem \ref{theorem:GeoQuoCox} and \S5.2 of \cite{david2011toric}) we have that any singular point $p\in \tilde{V}$ has a corresponding singular point in $V$, provided that $p\notin V(B) \subset\CC^m$. Hence, again using the toric ideal-variety correspondence, we conclude that the singularity subscheme of $V$ is the subscheme of $X_{\Sigma}$ defined by the $B$-saturated ideal $(I+K):B^{\infty}$. 
\end{proof}

\end{appendices}

\begin{footnotesize}
\bibliographystyle{plain}

\end{footnotesize}
\end{document}